\documentclass{article}

\usepackage{enumerate}
\usepackage{expl3}
\usepackage{setspace}
\usepackage[margin=3cm]{geometry}

\usepackage{graphicx}
\usepackage{tikz-cd}
\usepackage[all,cmtip]{xy}

\usepackage{amsmath}
\usepackage{amssymb}
\usepackage{amsthm}
\usepackage{mathtools}

\usepackage[
backend=biber,
style=alphabetic,
sorting=ynt
]{biblatex}
\usepackage[pdfencoding=auto,psdextra,colorlinks=true,linkcolor=black,citecolor=blue,urlcolor=blue]{hyperref}

\newtheorem{theorem}{Theorem}[section]
\newtheorem{lemma}[theorem]{Lemma}
\newtheorem{proposition}[theorem]{Proposition}
\newtheorem{corollary}[theorem]{Corollary}
\theoremstyle{definition}
\newtheorem{definition}[theorem]{Definition}
\newtheorem{terminology}[theorem]{Terminology}
\newtheorem{remark}[theorem]{Remark}
\newtheorem{example}[theorem]{Example}
\newtheorem{notation}[theorem]{Notation}

\DeclareMathOperator{\Ad}{Ad}
\DeclareMathOperator{\Alg}{Alg}

\DeclareMathOperator{\End}{End}
\DeclareMathOperator{\GL}{GL}

\DeclareMathOperator{\HAlg}{HAlg}
\DeclareMathOperator{\Hol}{Hol}
\DeclareMathOperator{\Id}{Id}

\DeclareMathOperator{\Lie}{Lie}

\DeclareMathOperator{\ad}{ad}

\newcommand{\Bis}{\mathrm{Bis}}

\newcommand{\calG}{\mathcal{G}}
\newcommand{\calH}{\mathcal{H}}
\newcommand{\calI}{\mathcal{I}}
\newcommand{\calK}{\mathcal{K}}
\newcommand{\calM}{\mathcal{M}}
\newcommand{\calO}{\mathcal{O}}

\newcommand{\F}{\mathcal{F}}
\newcommand{\FHT}{\mathrm{FHT}}

\newcommand{\G}{\mathcal{G}}
\newcommand{\GammaPath}{\mathcal{P}\Gamma}
\newcommand{\grpd}{\rightrightarrows}
\newcommand{\hol}{\mathrm{hol}}
\newcommand{\HT}{\mathrm{HT}}

\newcommand{\Inn}{\mathrm{Inn}}

\newcommand{\into}{\hookrightarrow}
\newcommand{\inverse}{\mathbf{i}}

\newcommand{\m}{\mathbf{m}}

\newcommand{\R}{\mathbb{R}}
\newcommand{\SC}{\mathcal{SZ}}
\newcommand{\src}{\mathbf{s}}

\newcommand{\THT}{\mathrm{THT}}
\newcommand{\til}{\widetilde}
\newcommand{\trg}{\mathbf{t}}
\newcommand{\unit}{\mathbf{u}}

\title{On the holonomy of Lie algebroids}
\author{Pooja Joshi and Joel Villatoro}
\date{ }
\begin{document}

\maketitle
\begin{abstract}
We introduce a holonomy groupoid for Lie algebroids. This construction generalizes both the holonomy groupoid of a foliation and the adjoint representation of a Lie algebra. We prove that the resulting groupoid is longitudinally smooth and compute its Lie algebroid. When the  algebroid of the holonomy groupoid coincides with the given Lie algebroid, the holonomy groupoid is the canonical terminal integration: every source-connected integration admits a unique morphism into it inducing the identity infinitesimally. To construct the groupoid structure on the holonomy groupoid, we utilize the notion of the "flow product" of time-dependent sections of a Lie algebroid. This provides an alternative way to define the groupoid structure for the Weinstein groupoid.

\end{abstract}

\setcounter{tocdepth}{1}
\tableofcontents
\newpage

\section{Introduction}

In the most general sense, a foliation is a decomposition of a smooth manifold into a disjoint union of submanifolds called leaves.
There are a few different variants of the notion of a foliation in the literature, including regular foliations, singular foliations, and  Riemannian singular foliations. However, they all tend to require (or imply) that the manifold locally splits into a product of a leaf and a transverse structure.

The local splitting condition means that the leaves of a foliation are initial submanifolds and have a well-defined ``transverse geometry.''
In the context of foliations, the term \emph{holonomy} refers to the global structure of this transverse data.
The holonomy of a foliation encodes information about how the leaves twist around one another and plays a central role in classical results such as the Reeb stability theorem~\cite{ReebStability} and the Thurston stability theorem~\cite{ThurstonStability}.

The study of foliations and the study of Lie algebroids/groupoids are closely intertwined.
There are many connections between the two. The most direct relationship is that Lie algebroids give rise to singular foliations. In the other direction, regular foliations can be viewed as a special case of Lie algebroids.

It is well-known  that the holonomy of a regular foliation can be encoded in a groupoid called the holonomy groupoid of the foliation.
Classically, the holonomy groupoid is constructed by placing an equivalence relation on the set of paths tangent to the foliation.
The equivalence relation is called ``holonomy'', and two paths are considered holonomic if the transverse geometry is twisted in the same way along both paths.

Androulidakis and Skandalis~\cite{Androulidakis+Skandalis} introduced a generalization of the holonomy groupoid to singular foliations using an alternative construction based on bisubmersions.
Their construction overcame some of the technical difficulties that arose from the less rigid transverse geometry of singular foliations. However, a price paid for this generalization is that the holonomy groupoid of a singular foliation is not necessarily a Lie groupoid, since it fails to be a smooth manifold in general.
Still, it is rather close to being a smooth manifold. 
Debord~\cite{LongtitudinalSmoothnesDebord} showed that the holonomy groupoid of a singular foliation was fully smooth when restricted to a leaf (this is called ``longitudinal smoothness'').

Later, Garmendia and Villatoro~\cite{SingJoelAlf} developed a new construction of the Androulidakis-Skandalis holonomy groupoid, which more closely resembled the classical construction and was defined in terms of paths.
Crucial to this construction is the notion of a ``holonomy transformation'' (originally due to Androulidakis and Zambon~\cite{Androulidakis+Zambon}),  which is an equivalence class on the set of germs of diffeomorphisms between transverse slices to the foliation.

The goal of this article is to extend the work of~\cite{SingJoelAlf} to the setting of Lie algebroids. We construct  a holonomy groupoid for a Lie algebroid and establish several of its basic properties. In particular, we show that it is longitudinally smooth and compute its Lie algebroid.
As a small application, we will use the holonomy groupoid to establish a sufficient condition for the existence of a ``minimal integral'' of a Lie algebroid.

\subsection*{Summary of Construction}

The two main ingredients of our construction are  $\Gamma(A)$-paths and  holonomy transformations.
These are developed in Section~\ref{section:gamma(A)-homotopy}.

\subsubsection*{\texorpdfstring{$\Gamma(A)$-paths}{Gamma(A)-paths}}
Given a Lie algebroid $A \to M$, we define a $\Gamma(A)$-path to be a pair $(\alpha(t),p_0)$,  where $\alpha(t) \in \Gamma(A)$ is a compactly supported time-dependent section of $A$ for $t \in [0,1]$, and $p_0 \in M$ is a point.
We write $\GammaPath(A)$ to denote the set of all $\Gamma(A)$-paths. 
We call such pairs $(\alpha(t),p_0)$ \emph{paths} because each determines an  underlying classical path $\gamma(t)$ in $M$ induced by the flow along $\alpha(t)$.

Sections of $A$ induce a flow on the total space of the Lie algebroid and so for each $\Gamma(A)$-path $(\alpha(t),p_0)$ we can consider the flow $\Phi^t_\alpha \colon A \to A$ which is a Lie algebroid automorphism of $A$.
This is usually referred to as the ``adjoint flow'' of $\alpha(t)$. It is a generalization of the adjoint representation of a Lie algebra.

The space of $\Gamma(A)$-paths comes equipped with a natural groupoid operation which we call the "flow product." We use this instead of the usual concatenation approach because it is well behaved with respect to the notion of holonomy. To our knowledge, the formula for the flow product is due to Duistermaat and Kolk~\cite{Lie-groups-Duistermaat},  but our use of it in the groupoid context is novel.

\subsubsection*{Holonomy transformations}
A holonomy transformation is an equivalence class of a germ of Lie algebroid isomorphisms  restricted to a slice.
The notion of a holonomy transformation is an adaptation of a closely related construction arising in the theory of singular foliations.
The singular foliation version was introduced by Androulidakis and Zambon~\cite{Androulidakis+Zambon}.

Given a point $x \in M$ and a slice $S_x$ through $x$, we can consider the restricted algebroid $A_{S_x} \to S_x$, which is a Lie algebroid over the slice.
Given a slice $S_x$ through $x$ and a slice $S_y$ through $y$, we can consider the germs of Lie algebroid isomorphisms from $A_{S_x}$ to $A_{S_y}$ which map $x$ to $y$.
The collection of all such germs is called the \emph{full holonomy transformation groupoid} and is denoted $\FHT(A)$. 
The objects of $\FHT(A)$ are slices through points in $M$, and the morphisms are germs of Lie algebroid isomorphisms between the corresponding slice algebroids.

A holonomy transformation is called \emph{trivial} if it is the time-1 flow of a time-dependent family of sections of $A$ which vanish at $x$. The
trivial holonomy transformations form a normal subgroupoid of $\FHT(A)$, which we denote by $\THT(A)$.
The orbits of $\THT(A)$ consists of all slices through a given point $x$. 
From this, one obtains the \emph{reduced holonomy transformation groupoid} $\HT(A)$, defined as the quotient of $\FHT(A)$ by $\THT(A)$. Its space of objects is naturally identified with $M$. 

\subsubsection*{Holonomy homomorphism}

The set of $\Gamma(A)$-paths carries a natural groupoid structure over $M$ arising  from the composition of  adjoint flows. 
Furthermore, given a $\Gamma(A)$-path $(\alpha(t),p_0)$,  the adjoint flow defines a germ of a Lie algebroid automorphism around $p_0$ and hence an element of $\HT(A)$.
This assignment is a groupoid homomorphism, which we call the \emph{holonomy homomorphism}:
\[ \Hol \colon \GammaPath(A) \to \HT(A). \]
Two $\Gamma(A)$-paths are said to be \emph{holonomic} if they have the same image under the holonomy homomorphism.
The holonomy groupoid of $A$ is the quotient of $\GammaPath(A)$ by the holonomy equivalence relation,  and is denoted $\Hol(A)$.

\subsubsection*{Groupoid picture}
We can also  define holonomy for a general Lie groupoid. 
This notion coincides with the holonomy of an algebroid in the sense that the holonomy of a source-connected groupoid agrees with the holonomy of its underlying Lie algebroid.

Given an element $g \in \calG$ of a Lie groupoid, a local bisection $\sigma$ through $g$ defines a germ of a Lie algebroid isomorphism from $A$ near the source of $g$ to $A$ near the target of $g$.
Such a Lie algebroid isomorphism defines an element of $\HT(A)$, and this element is independent of the choice of local bisection. 
From this, one obtains a groupoid homomorphism from $\calG$ to $\HT(A)$, and two elements of $\calG$ are said to be holonomic if they have the same image under this homomorphism.
The holonomy groupoid of $\calG$ is the quotient of $\calG$ by the holonomy equivalence relation and is denoted $\Hol(\calG)$.

The groupoid notion of holonomy is compatible with the corresponding algebroid notion. More precisely, if $\calG$ is a source-connected Lie groupoid with Lie algebroid $A$, then $\Hol(\calG)$ is canonically isomorphic to $\Hol(A)$ (see Corollary~\ref{Cor-Groupoid-Hol-Equal-Algeb-Hol})

\subsubsection*{Algebroid of the holonomy groupoid}
Since the holonomy groupoid of a Lie algebroid is longitudinally smooth, its restriction to a single leaf has  a well-defined Lie algebroid. 
This Lie algebroid can be constructed directly from the original Lie algebroid. It is called the \emph{holonomy Lie algebroid} of $A$, denoted by $\HAlg(A)$.

This object can be constructed directly using the notion of a \emph{strongly central element} of a Lie algebroid.
Given $a \in A$, we say that $a$ is strongly central if there exists a section $\alpha \in \Gamma(A)$ that extends $a$ and takes values only in the centers of the isotropy Lie algebras of $A$.
We write $\SC(A)$ to denote the set of all strongly central elements of $A$.

The holonomy algebroid of $A$ is the quotient of $A$ by the strongly central elements and is denoted by $\HAlg(A)$. 
Calling $\HAlg(A)$ a Lie algebroid is a bit of an abuse of terminology,  since the rank of $\SC(A)$ is not constant.
It would be more correct to regard $\HAlg(A)$ as a singular algebroid in general.
Nevertheless, like the holonomy groupoid, $\HAlg(A)$ is longitudinally smooth in the sense that its restriction to any leaf is  a Lie algebroid.

\subsection*{Main Results}
Apart from the construction of the holonomy groupoid, the principal technical results of this article are the following:
\begin{enumerate}
\item The holonomy groupoid $\Hol(A)$ is longitudinally smooth. That is, it is a Lie groupoid when restricted to a single orbit (Theorem~\ref{theorem:hol(A)-is-long-smooth}).
\item Given a leaf $L$ of $M$, the Lie algebroid of $\Hol(A)|_L$ is canonically isomorphic to the holonomy algebroid $\HAlg(A)|_L$ (Theorem~\ref{theorem:holonomy-algebroid-isomorphism}).
\item If the set of strongly central elements of $A$ is trivial, then the holonomy groupoid $\Hol(A)$ is an adjoint integration of $A$ (Theorem~\ref{theorem:holonomy-adjoint-integration}). 
By this we mean, given a source-connected Lie groupoid $\calG$ with Lie algebroid $A$, there exists a unique groupoid homomorphism $\calG \to \Hol(A)$ which is the identity on $A$.
\end{enumerate}
The last point means that, under appropriate circumstances, the holonomy groupoid can serve as a ``terminal'' integration of the Lie algebroid.

In addition to the main results, Section~\ref{section:examples} contains a collection of examples illustrating the scope of the construction. In particular, the holonomy groupoid recovers the classical adjoint representation for Lie algebras, the pair groupoid for the tangent algebroid, and the usual holonomy groupoid of a foliation. Further examples show how the construction encodes isotropy representations for transitive and action Lie algebroids and specializes naturally to regular Poisson manifolds.

\subsection*{Acknowledgements}

Pooja Joshi was supported by the National Science Foundation (Award Number DMS-2350181). She would like to thank her Ph.D. advisor, Xiang Tang, for  encouraging her to pursue this project and for his continued support. 
\\
\noindent Joel Villatoro would like to thank Marco Zambon for hosting him and numerous discussions when this project was initiated. He was funded in part by the following sources while preparing this article: Fonds Wetenschappelijk Onderzoek (FWO Project G083118N); the National Science Foundation (Award Number 2137999).

\newpage
\section{Preliminaries and Notation}
\subsection{Vector bundles}
\begin{notation}
    Given a vector bundle automorphism \( F \colon E \to E \) covering a diffeomorphism \( f \colon M \to M \), we write:
    \[ F_* \colon \Gamma(E) \to \Gamma(E), \qquad (F_* \eta)_p := F(\eta_{f^{-1}(p)})\]
    to denote the push-forward of sections of \( E \) along \( F \).
\end{notation}
\begin{notation}
    Given a vector bundle \( E \to M \) and a section \( \eta \in \Gamma(E) \) we write \( \eta_p \in E_p \) to denote the value of \( \eta \) at a point \( p \in M \).
\end{notation}
\begin{notation}
    Given a Lie group $G$, we write $\mathfrak{g}$ to denote its Lie algebra.
\end{notation}

\subsection{Lie algebroids and Lie groupoids}

\begin{notation}
In this article, we will write $\calG \grpd M$ to mean that $\calG$ is the set of arrows for a groupoid with objects $M$. Most often we will be considering Lie groupoids. As is typical, we do not require the object manifold $\calG$ to be Hausdorff.

We will write $\src \colon \calG \to M$ and $\trg \colon \calG \to M$ to denote the source and target map of a groupoid. Given an object $x \in M$,  we write $1_x \in \calG$ to denote the associated unit arrow. We may write $1_M \subset \calG$ to denote the entire submanifold of unit elements. The multiplication operation as a map will be denoted by $\m$ but multiplying elements will usually be expressed using a small dot or just concatenation:
\[ \m(g,h) = g \cdot h = gh.\]
The inverse map will be denoted by $\inverse$ and inverses of elements will be expressed as negative powers $\inverse(g) = g^{-1}$.
\end{notation}

\begin{notation}
    We will typically denote Lie algebroids in terms of the total space (usually $A$). Given a Lie algebroid $A$, the anchor map will be denoted by $\rho_A \colon A \to TM$ and the Lie bracket will be denoted by $[\alpha,\beta]_A$. The $A$ superscript may be omitted when it is clear from context, which will be most of the time.
\end{notation}

Lie algebroids are the infinitesimal version of a Lie groupoid. Our convention will be to identify the Lie algebroid of a Lie groupoid with the restriction of the source distribution to the units.

\begin{definition}
    Suppose $\calG \grpd M$ is a Lie groupoid. The \emph{source distribution} is the kernel of $d s \colon T\calG \to TM$,  and is denoted by  $T^s \calG$. The \emph{Lie algebroid associated to $\calG$} is defined to be:
    \[ \Lie(\calG) := T^s\calG|_{1_M}.\]
    We write $\Lie(\calG)$ for the Lie algebroid associated to $\calG$.
    
\end{definition}

Under this convention, the Lie bracket on sections of $A$ is induced by the Lie bracket on \emph{right}-invariant sections of $T^s \calG$.

The operation $\Lie$ is actually a functor from the category of Lie groupoids to the category of Lie algebroids. Given a smooth groupoid homomorphism $F \colon \calG \to \calH$, define
\[ \Lie (F) := d F|_{\Lie(\calG)}.\]

\section{Holonomy of a Lie algebroid}\label{section:gamma(A)-homotopy}

The algebroid holonomy groupoid of a Lie algebroid simultaneously generalizes the holonomy of a foliation and the  adjoint representation of a Lie algebra.
It is a groupoid that encodes information about the holonomy of the characteristic foliation in a way that keeps track of how the holonomy interacts with the isotropy Lie algebras.

However, in order to understand this notion, we first need to discuss the topic of the ``adjoint flow'' of a Lie algebroid. This can be found in previous literature \cite{Rui+Marius}  as the ``flow of a section.''

To establish some notation throughout this section and for the rest of the paper: Given a vector bundle $E \to M$,  we write $\Gamma_c(E)$ to denote the vector space of compactly-supported sections of $E$. We will write $\Gamma^t_c(E)$ to denote the set of time-dependent compactly-supported sections:
\[ \Gamma^t_c(E) := \{ \eta \colon [0,1] \to \Gamma_c(E)  \}. \]
When we discuss parameterized sections that are compactly supported, we always assume that the compact support is \emph{uniform} in the parameter, meaning that a single compact set works for all values of the parameter.
\subsection{Adjoint flows}

\begin{definition}
    Suppose \( \alpha (t) \in \Gamma(A)  \) is a time-dependent section of \( A \). An \emph{adjoint flow equation} for \( \alpha \) is an initial value problem of the form:
    \begin{equation}\label{eqn:adjoint.flow.equation} 
        \frac{d}{dt} \beta = [\beta, \alpha(t)] \qquad \beta(0) = \beta_0.
    \end{equation}
    For some initial condition \( \beta_0 \in \Gamma(A) \).

    The \emph{adjoint flow} of \( \alpha(t) \), written \( \Phi^t_\alpha \) is the unique one parameter family of Lie algebra automorphisms which satisfies:
    \begin{equation}\label{eqn:adjoint.flow} 
        \frac{d}{dt} \Phi^t_\alpha = [\Phi^t_\alpha, \alpha(t)] \qquad \Phi^0_\alpha = Id_A.
    \end{equation}
    The base map of such Lie algebra automorphisms is \( \phi^t_X \colon M \to M \) and is given by the flow of \( X = \rho(\alpha(t)) \).
\end{definition}
Our first lemma tells us that adjoint flows always exists so long as the vector field \( \rho(\alpha(t)) \) is complete. The argument relies on the theory of vector bundle derivations (see Appendix~\ref{appendix:derivations} for a brief version).
\begin{lemma}
    Suppose \( \alpha(t) \in \Gamma(A) \) is a time-dependent section of \( A \) and \( X(t) := \rho(\alpha(t)) \) is a complete vector field. Then one-parameter family of vector bundle automorphisms satisfying Equation~\ref{eqn:adjoint.flow} exists for all time.
\end{lemma}
\begin{proof}
    \( \alpha(t) \) defines a time-dependent family of derivations on \( A \) by taking:
    \[ D(t) \colon \Gamma(A) \to \Gamma(A) \qquad D(\beta) = [\alpha(t),\beta]. \]
    The symbol of this derivation is \( X(t) \). Since \( X(t) \) is complete it follows that there exists a one-parameter family of vector bundle automorphisms \( \Phi^t_\alpha \) satisfying:
    \[ \frac{d}{dt} \Phi^t_\alpha = - D(t) \circ \Phi^t_\alpha \qquad \Phi^0_\alpha = Id_A. \]
    However, this is the same equation as Equation~\ref{eqn:adjoint.flow}.

Using lemma \ref{lemma:algebroidflow:morphism}, we prove that $\Phi^t_\alpha$ is a Lie algebroid automorphism. 
   
\end{proof}
 
\begin{example}
Suppose \( A = TM \) is the tangent algebroid of a manifold \( M \). If \( X(t) \) is a complete, time-dependent vector field, let \( \phi^t_X \colon M \to M \) denote the flow of \( X \). A standard computation shows that:
\[ \frac{d}{dt} d\phi^t_X = [d\phi^t_X, X], \qquad d \phi^0_X = Id_{TM}. \]
Hence the adjoint flow \( \Phi^t_X \) of \( X \) is the same as \( d \phi^t_X \).
\end{example}
\subsection{\texorpdfstring{\(\Gamma(A)\)-paths}{Gamma(A)-paths}}
The construction of the holonomy groupoid is very similar in spirit to the construction which appears in \cite{SingJoelAlf}.
The idea is to measure the action of the adjoint flow of a section on the transverse part of the Lie algebroid by using time-dependent families of sections.
\begin{definition}
    Suppose \( A \) is a Lie algebroid. A \( \Gamma(A) \)-path consists of a pair \( (\alpha, p_0) \in \Gamma^t_c(A) \times M \). In other words, \( \alpha(t) \) is a compactly supported, time-dependent section of \( A \) and \( p_0 \in M \) is a point in \( M \). We denote the set of all \( \Gamma(A) \)-paths by  
    \[ \GammaPath(A) = \Gamma^t_c(A) \times M. \]
\end{definition}

\begin{definition}
    Suppose \( a(t) \) is an \( A\)-path. We say that a time-dependent section \( \alpha(t) \) \emph{extends} \( a(t) \) if  \( \alpha(t)_{p(t)} = a(t) \), where \( p(t) = \pi(a(t)) \) is the underlying path of \( a(t) \).
\end{definition}
\( \Gamma(A) \)-paths are closely related to the classical notion of an \( A \)-path but they are not quite the same.
In general, given a \( \Gamma(A) \)-path, \( (\alpha(t),p_0) \), it is possible to construct an \( A \)-path by taking:
\[ a(t) := \alpha(t)_{p(t)} \]
where \( p(t) \) is the unique integral curve of \( \rho(\alpha(t)) \) with initial condition \( p(0) = p_0 \).
\subsection{Homotopies of \texorpdfstring{\( \Gamma(A) \)-paths}{Gamma(A)-paths}}
The notion of \(A \)-homotopy of \( A \)-paths can be extended to a notion of homotopy for \( \Gamma(A) \)-paths. For this we will discuss two parameter families of sections of $A$. To set some notation, given a two-parameter family of sections $\alpha(t,s) \in \Gamma(A)$ we will write:
\[ \Phi^{t,s}_\alpha \colon \Gamma(A) \to \Gamma(A) \]
to denote the flow of $\alpha(t,s)$ \emph{in the $t$-direction}.

For the underlying flow in the manifold we will write:
\[ \phi^{t,s}_\alpha \colon M \to M.\]

To understand when a two-parameter family of sections constitutes a homotopy, we need the notion of a complement.
\begin{definition}
    Suppose \( \alpha(t,s) \in C^\infty([0,1]^2,\Gamma(A)) \) is a two parameter family of sections of \( A \). The \emph{complement} of \( \alpha \) is the unique two parameter family of sections \( \beta(t,s) \in C^\infty([0,1]^2, \Gamma(A)) \) which satisfies the following initial value problem:
    \[ \frac{\partial \beta}{\partial t} - \frac{\partial \alpha}{\partial s}  = [\beta, \alpha], \qquad \beta(0,s) = 0. \]
  
\end{definition}
Intuitively, when one has a two parameter family of sections \( \alpha \) where we think of \( \alpha \) as representing the derivative in the ``\( t \)-direction'' the complement is the corresponding derivative in the ``\( s \)-direction.''
The complement of a two parameter family of sections always exists so long as \( \alpha(t,s) \) is compactly supported since it can be constructed explicitly via the adjoint flow of \( \alpha \) (see Lemma \ref{lemma:complement.of.two.parameter.family}):
\begin{equation}\label{eqn:compliment.of.two.parameter.family.formula} 
\beta(t,s) = \Phi^{t,s}_\alpha \int_0^t (\Phi^{u,s}_\alpha)^{-1} \left( \frac{\partial \alpha}{\partial s} (u,s) \right) du. 
\end{equation}
The complement is a crucial ingredient for defining homotopies of \( \Gamma(A) \)-paths.
\begin{definition}
Suppose $\alpha_0$ and $\alpha_1 \in \Gamma^t_c(A)$ are two time-dependent sections of $A$. Let $p \in M$ be fixed. We say that a two parameter family $\alpha(t,s)$ which interpolates $\alpha_0$ and $\alpha_1$ is a \emph{homotopy at $p$} if the complement $\beta(t,s)$ satisfies:
\[ \beta(t,s)_{\phi^{t,s}_\alpha(p)} = 0.\]

In such a case, we say that the $\Gamma(A)$-paths $(\alpha_0,p)$ and $(\alpha_1,p)$ are homotopic.
\end{definition}
Note that if $(\alpha_0, p)$ is homotopic to $(\alpha_1,p)$ and $\gamma_0(t) = \phi_{\alpha_0}^t(p)$ and $\gamma_1(t) = \phi^t_{\alpha_1} (p)$ are the associated underlying paths, then $\gamma_0$ is classically homotopic to $\gamma_1$. This is a consequence of compatibility between the anchor map and the notion of the complement:
$$\frac{\partial \gamma}{\partial s} (t,s) = \rho(\beta(t,s))_{\gamma(t,s)}. \qquad \text{(see Lemma~\ref{lemma:flows.of.complements.for.integral.curves})} $$ 

This homotopy equivalence relation corresponds precisely with the classical notion of homotopy for \( A \)-paths. In the following sense: If \( a_0(t) \) and \( a_1(t) \) are homotopic \( A \)-paths then any pair of extensions \( \alpha_0(t) \) and \( \alpha_1(t) \) of \( a_0(t) \) and \( a_1(t) \) respectively are homotopic as \( \Gamma(A) \)-paths(see Lemma \ref{lemma:A:homotopypaths:Gamma(A)paths}). Conversely, if \( (\alpha(t,s),p_0) \) is a homotopy between two \( \Gamma(A) \)-paths then the underlying \( A\)-paths are homotopic as well (see Lemma  \ref{lemma:A:homotopypaths:Gamma(A)paths})

\begin{example}\label{example:reparameterization.as.homotopy}
Suppose $\rho \colon [0,1] \to [0,1]$ is a smooth function that is the identity on the endpoints and suppose $\alpha_0(t)$ is a time-dependent section of $A$. 
Then we can construct another section:
$$\alpha_\rho(t) = \rho'(t) \alpha_0(\rho(t)) $$
which we call the reparameterization of $\alpha_0$ by $\rho$. Such a reparameterization is homotopic to $\alpha_0$ at all points $p \in M$. Indeed, all reparameterizations are homotopic to one another.

To see why, consider a smooth function $\rho(t,s) \colon [0,1]^2 \to [0,1]$ which satisfies $\rho(0,s) = 0$ and $\rho(1,s) = 1$ for all $s \in [0,1]$. 
Then we can define a two parameter family of sections:
$$ \alpha(t,s) = \frac{\partial \rho}{\partial t} (t,s) \alpha_0(\rho(t,s)) $$
This $\alpha$ interpolates between the reparameterization of $\alpha_0$ by $\rho(t,0)$ and the reparameterization of $\alpha_0$ by $\rho(t,1)$.

A direct calculation shows that the complement of $\alpha$ is given by:
$$ \beta(t,s) = \frac{\partial \rho}{\partial s} (t,s) \alpha_0(\rho(t,s)) $$
Since $\partial_s \rho (1,s) = 0$ it follows that $\beta(1,s) = 0$ and hence $\alpha$ constitutes a homotopy. This homotopy is valid for all points since $\beta(1,s)$ vanishes everywhere.
\end{example}


\subsection{Holonomy of a \texorpdfstring{\( \Gamma(A) \)-path}{Gamma(A)-path}}Let us now define the holonomy equivalence relation. In order to do this, we must introduce the notion of the holonomy transformation groupoid. To begin, let us define the notion of a slice.

\begin{definition}
    Let \( x \in M \). A \emph{slice through \( x \)} is an embedded submanifold \( S \subset M \) such that,  for all \( y \in S \),  we have that:
    \begin{itemize}
 \item \(  T_y M = T_y S + \rho(A_y) \)
 \item \( T_x S \cap \rho(A_x) = 0 \)
    \end{itemize}
    Given such a slice \( S \) the \emph{slice algebroid} is the Lie algebroid \( A_S \to S \) defined by:
    \[ A_S = \rho^{-1}(TS). \]
\end{definition}
A slice through \( x \) intersects the characteristic foliation in a clean transversal at \( x \) and is transversal to the characteristic foliation at every point of the slice. In general, the transversality will not be clean at points other than \( x \) unless the characteristic foliation is regular. Note that the slice algebroid is not the same as restricting the original algebroid to the slice as a vector bundle as this would not be a Lie algebroid in general. The well-definedness of the slice algebroid is a consequence of the transversality condition (see Lemma~\ref{lemma:slice.algebroid.well.defined}).
\begin{definition}
    A \emph{holonomy transformation} from a slice \( S_x \) through \( x \) and a slice \( S_y \) through \( y \) is a germ of a Lie algebroid isomorphism from \( A_{S_x} \) to \( A_{S_y} \) around \( x \in S_x \). We write \( \Alg(A_{S_x},A_{S_y}) \) to denote the set of all holonomy transformations from \( A_{S_x} \) to \( A_{S_y} \).
    
    The \emph{full holonomy transformation groupoid}, denoted \( \FHT(A) \), is the groupoid consisting of all holonomy transformations between all slices through points in \( M \):
    \[ \FHT(A) := \bigsqcup_{x,y \in M} \bigsqcup_{S_x,S_y} \Alg(A_{S_x}, A_{S_y}) \]
    where \( S_x \) and \( S_y \) range over all slices through \( x \) and \( y \) respectively.

    We say a holonomy transformation \( \Theta \in \FHT(A) \) is \emph{trivial} if it is the time-1 flow of a time-dependent family of sections of \( A \) which vanish at \( x \). We write \( \THT(A) \) to denote the normal subgroupoid of \( \FHT(A) \) consisting of all trivial holonomy transformations.
    \[ \THT(A) := \bigsqcup_{x \in M} \bigsqcup_{S_x} \{ \Theta \in \FHT(A) \ : \ \Theta = \Phi^1_\alpha \text{ for } \alpha(t) \in I_x \Gamma(A) \}. \]

    The \emph{(reduced) holonomy transformation groupoid}, written \( \HT(A) \), the quotient groupoid:
    \[ \HT(A) := \FHT(A) / \THT(A). \]
    \begin{example}
        The full holonomy transformation groupoid and the reduced holonomy transformation groupoids of a Lie algebra are the same, i.e., equal to $Aut(\mathfrak{g})$.
        
    \end{example}

Note that, since every pair of slice algebroids through a point can be related by a trivial holonomy transformation, the objects of \( \HT(A) \) can be identified with points in \( M \).

See Lemma~\ref{lemma:trivial.holonomy.transformations.are.normal.subgroupoid} for for a proof that \( \THT(A) \) is a normal subgroupoid of \( \FHT(A) \) and hence the quotient \( \HT(A) \) is well-defined. See Lemma~\ref{lemma:slice.algebroids.connected.by.trivial.holonomy} for a proof that any pair of slice algebroids through the same point can be connected by a trivial holonomy transformation.
\end{definition}
\begin{lemma}
    Suppose \( \Phi \in \Alg_x(A,A) \) is a germ of a Lie algebroid isomorphism around \( x \in M \). Then for any slice \( S_x \) through \( x \in M \) we have that \( \Phi|_{S_x} \colon A_{S_x} \to A_{\Phi(x)} \) defines an element of \( \HT(A) \). We claim that the equivalence class of this element does not depend on the choice of slice \( S_x \).
\end{lemma}
\begin{proof}
    Let \( S_x \) and \( S_x' \) be two slices through \( x \). Write \( S_y = \Phi(S_x) \) and \( S'_y = \Phi(S_x') \) to denote the slices through \( \Phi(x) \). 
    We need to show that the holonomy transformation \( \Phi|_{S_x} \colon A_{S_x} \to A_{S_y} \) is equivalent to the holonomy transformation \( \Phi|_{S_x'} \colon A_{S_x'} \to A_{S'_y} \) in \( \HT(A) \). By Lemma~\ref{lemma:slice.algebroids.connected.by.trivial.holonomy} there exists a time-dependent family of sections \( \alpha(t) \in \Gamma(A) \) with the property that the germ of \( \Theta := \Phi^1_\alpha|_{S_x} \) around \( x \) defines a trivial holonomy transformation from \( S_x \) to \( S_x' \). Therefore, we obtain a commutative diagram:
    \[
    \begin{tikzcd}
        S_x \arrow[r, " \Phi|_{S_x} "] \arrow[d, "\Theta"'] & S_y \arrow[d, "\Phi \Theta \Phi^{-1}"] \\
        S_x' \arrow[r, " \Phi|_{S_x'} "'] & S'_y
    \end{tikzcd}
    \]
    Since \( \THT(A) \) is a normal subgroupoid (Lemma~\ref{lemma:trivial.holonomy.transformations.are.normal.subgroupoid}), it follows that \( \Phi \Theta \Phi^{-1} \) is a trivial holonomy transformation from \( S_y \) to \( S'_y \). 
    Hence, the holonomy transformation \( \Phi|_{S_x} \) is equivalent to the holonomy transformation \( \Phi|_{S_x'} \) in \( \HT(A) \).
\end{proof}
Due to the above lemma, one can define an element of \( \HT(A) \) from a germ of a Lie algebroid isomorphism around a point \( x \) without needing to specify a slice through \( x \). This means that the following definition makes sense.
\begin{definition}
We say that two \( \Gamma(A) \)-paths \( (\alpha(t),p_0) \) and \( (\beta(t),p_0) \) are \emph{holonomic} if the time-1 flow of \( (\alpha(t),p_0) \) and the time-1 flow of \( (\beta(t),p_0) \) define the same element of \( \HT(A) \).

The \emph{algebroid holonomy groupoid} of \( A \), written \( \Hol(A) \grpd M \), is the groupoid whose objects are points in \( M \) and whose morphisms from \( x \) to \( y \) are holonomy equivalence classes of \( \Gamma(A) \)-paths from \( x \) to \( y \). We write:
\[ \Hol \colon \GammaPath(A) \to \Hol(A) \]
to denote the quotient map which sends a \( \Gamma(A) \)-path to its holonomy equivalence class.

\end{definition}

\begin{example}
    For $A = \mathfrak{g} \rightarrow \{\ast\}$, $\Hol(A)$ recovers $\Inn(\mathfrak{g}),$
as in Example \ref{example:Lie-algebra}. For $A=TM$, it recovers the pair groupoid on connected components, as in Example \ref{example:holonomy-algebroid-tangent-bundle}. For $A=\F\subset TM$, it recovers the usual holonomy groupoid of the foliation, as in Example \ref{example:algebroid-holonomy-foliations}.

\end{example}

\subsection{Flow product groupoid structure}
In the previous subsection, we defined the holonomy groupoid only as a set but we still need to explain how it obtains a groupoid structure. The approach we will take is (somewhat) new and is based on the fact that the space of time-dependent sections of \( A \) has a natural group structure, which we call the ``flow product''.

The formula we use for the flow product is originally due to Duistermaat and Kolk~\cite{Lie-groups-Duistermaat}. They used it in the context of finite dimensional Lie groups. However, it turns out that it works well in this infinite dimensional context and provides us with a nice way to define the groupoid structure on $\Pi_1(A)$ and $\Hol(A)$. 
\begin{definition}
Suppose \( \alpha, \beta \in \Gamma^t_c(A)\) are compactly supported time-dependent sections of \( A \). The \emph{flow product} of \( \alpha \) and \( \beta \) is the time-dependent section \( \alpha \bullet \beta \) defined by:
\[ \alpha \bullet \beta (t) := \alpha(t) + (\Phi^t_\alpha)_* \beta(t). \]
\end{definition}
This definition is motivated by the following fact (see Lemma~\ref{lemma:flow.product.compatible.with.flows}):
\begin{equation}\label{eqn:flow.product} \Phi^t_{\alpha \bullet \beta} = \Phi^t_\alpha \circ \Phi^t_\beta. \end{equation}
One important consequence of this is the observation that the flow product forms an honest group structure.
\begin{lemma}
The flow product makes $\Gamma^t_c(A)$ into a diffeological group.
\end{lemma}
\begin{proof}
    The identity element is the zero section and the inverse of a time-dependent section \( \alpha(t) \) is given by:
    \[ \alpha^{-1}(t) = - (\Phi^t_\alpha)_*^{-1} \alpha(t). \]
    Let us verify that the flow product is associative. Suppose \( \alpha(t) \), \( \beta(t) \), and \( \eta(t) \) are compactly supported, time-dependent sections of \( A \). We need to show that:
\[ (\alpha \bullet \beta) \bullet \eta (t) = \alpha \bullet (\beta \bullet \eta)(t). \]
By a direct computation:
\begin{align*}
    (\alpha \bullet \beta) \bullet \eta (t) & = (\alpha \bullet \beta)(t) + (\Phi^t_{\alpha \bullet \beta})_* \eta(t) \\
    & = \alpha(t) + (\Phi^t_\alpha)_* \beta(t) + (\Phi^t_{\alpha \bullet \beta})_* \eta(t) \\
    & = \alpha(t) + (\Phi^t_\alpha)_* \beta(t) + (\Phi^t_\alpha \circ \Phi^t_\beta)_* \eta(t) \\
    & = \alpha(t) + (\Phi^t_\alpha)_* (\beta(t) + (\Phi^t_\beta)_* \eta(t)) \\
    & = \alpha \bullet (\beta \bullet \eta)(t).
\end{align*}
\end{proof}
Similarly, the flow product makes \( \Gamma(A) \)-paths into a diffeological \emph{groupoid}. Indeed, it can be thought of as an action groupoid for the action of the group of time-dependent sections on \( M \).
\begin{definition}
The groupoid structure on \( \GammaPath(A) \) is defined as follows. The source and target maps are given by:
\[ \src(\alpha,p) := p, \qquad \trg(\alpha,p_0) :=  \phi^1_{\alpha}(p_0). \]
The unit map is given by:
\[ \unit(p) = (0,p). \]
the inverse map is given by:
\[ i(\alpha,p) = (- (\Phi^t_\alpha)_*^{-1} \alpha, \phi^1_\alpha(p)) \]
and the composition is given by the flow product. If \( (\alpha(t),p) \) and \( (\beta(t),q) \) are \( \Gamma(A) \)-paths where \( \phi^1_\beta (q) = p \) then the composition is given by:
\[ (\alpha,p) \bullet (\beta,q) = (\alpha \bullet \beta, q). \]
\end{definition}
The flow product is compatible with both the homotopy and holonomy equivalence relations.
\begin{proposition}\label{Proposition:Flow-product-compatible-holonomy}
    The flow-product is compatible with holonomy. In other words:
    \[ \Hol(( \alpha,p) \bullet (\beta,q)) = \Hol(\alpha,p) \cdot \Hol(\beta,q) \]
    whenever this expression is well-defined.
\end{proposition}
\begin{proof}
    Suppose we have \( \Gamma(A) \)-paths \( (\alpha(t),p_0) \) and \( (\beta(t),q_0) \) where the source of \( (\alpha(t),p_0) \) equals the target of \( (\beta(t),q_0) \) is \( q_0 \). 
    
    We need to show that the time-1 flow of \( \alpha \bullet \beta \) is equivalent to the composition of the time-1 flow of \( \alpha \) and the time-1 flow of \( \beta \) in the reduced holonomy transformation groupoid \( \HT(A) \). In fact, a stronger statement holds. By Equation~\ref{eqn:flow.product} we have that:
\[ \Phi^1_{\alpha \bullet \beta} = \Phi^1_\alpha \circ \Phi^1_\beta. \]
    This implies that the holonomy transformation induced by the germ of \( \Phi^1_{\alpha \bullet \beta} \) at \( q_0 \) is the same as the composition of the holonomy transformation defined by the germ of \( \Phi^1_\alpha \) around \( p_0 \) composed with the holonomy transformation defined by \( \Phi^1_\beta \) at \( q_0 \). 
\end{proof}

\begin{proposition}\label{proposition:flow.product.compatible.with.homotopy}
    The flow product is compatible with homotopy. 
    In other words, if \( (\alpha(t,s),p_0) \) and \( (\alpha'(t,s),q_0) \) are homotopies where the source of \( (\alpha(t,s),p_0) \) equals the target of \( (\alpha'(t,s),q_0) \) is \( q_0 \), then \( ((\alpha \bullet \alpha')(t,s),q_0) \) is a homotopy as well.
\end{proposition}
\begin{proof}
    According to Lemma~\ref{lemma:complement.of.flow.product.is.flow.product.of.complements} the complement of \( \alpha \bullet \alpha' \) is 
    \[ Z := \beta + \Phi^t_{\alpha}  \beta' \] 
    where \( \beta \) and \( \beta' \) are the complements of \( \alpha \) and \( \alpha' \) respectively. 
    
    Let \( \phi_\alpha^t \colon M \to M \) and \( \phi^t_{\alpha'} \colon M \to M \) be the flows of \( \rho(\alpha) \) and \( \rho(\alpha') \), respectively. The assumption that \( (\alpha, p_0) \) and \( (\alpha', q_0) \) are composable amounts to \( \phi^1_{\alpha'} (q_0) = p_0 \). That these two families constitute homotopies means that:
\[ \beta(1,s)_{\phi^1_{\alpha}(p_0)} = 0, \qquad \beta'(1,s)_{p_0} = 0. \]
    In order for \( \alpha \bullet \alpha' \) to be a homotopy we must have that:
\[ Z(1,s)_{\phi^1_{\alpha}(p_0)} = 0. \]
Computing this from the definition gives:
\begin{align*}
    Z(1,s)_{\phi^1_{\alpha}(p_0)} & = \beta(1,s)_{\phi^1_{\alpha}(p_0)} + (\Phi^1_\alpha)_* \beta'(1,s)_{\phi^1_{\alpha}(p_0)} \\
    & = 0 + \Phi^1_\alpha( \beta'(1,s)_{p_0} ) \\
    &= 0.
\end{align*}
Hence \( \alpha \bullet \alpha' \) is a homotopy.
\end{proof}
\subsection{Properties of the algebroid holonomy}

We now establish some basic properties of the algebroid holonomy groupoid. The first result shows that the holonomy relation is insensitive to homotopies of $\Gamma(A)$-paths. As a consequence, the holonomy map factors through the Weinstein groupoid.

\begin{proposition}\label{Proposition:A-homotopic-holonomic}
    If two $\Gamma(A)$-paths are homotopic, then they are holonomic. 
\end{proposition}
\begin{proof}
    Suppose $(\alpha_0, p_0)$ and $(\alpha_1, p_0)$ are homotopic $\Gamma(A)$-paths. Let $\alpha(t,s)$ be a homotopy between them and let $\beta(t,s)$ be its complement. 

    By definition of a $\Gamma(A)$-homotopy,
\[ \beta(1,s) = 0 \quad \quad \forall s \in [0,1].\]
Therefore, Lemma \ref{Lemma:adjoint-flows-are-equal} implies that the corresponding time -1 adjoint flows agree \[ \Phi^1_{\alpha_0} = \Phi^1_{\alpha_1}.\]
Hence $(\alpha_0, p_0)$ and $(\alpha_1, p_0)$ define the same element of the holonomy groupoid, so they are holonomic.
    
\end{proof}

The following theorem shows how algebroid holonomy relates to the existing groupoids associated to a Lie algebroid.
\begin{theorem}\label{theorem:holonomy.groupoid.sits.in.between.weinstein.and.foliation.holonomy}
There is a sequence of smooth groupoid homomorphisms:
\[ \GammaPath (A) \to \Pi_1(A) \to \Hol(A) \to \Hol(\mathcal{F} ). \]
In particular, the algebroid holonomy class of a \( \Gamma(A) \)-path depends only on the homotopy class of the underlying \( A \)-path.

\end{theorem}
\begin{proof}
To go from an element $(\alpha(t),p_0) \in \GammaPath(A)$ to an element of $\Pi_1(A)$ we simply take the $A$-homotopy class of the associated $A$-path $a(t) = \alpha(t)_{p(t)}$, where $p(t)$ is the integral curve of $\rho(\alpha(t))$ with initial condition $p(0) = p_0$. 

To go from an element of $[a(t)] \in \Pi_1(A)$ to an element of $\Hol(A)$ we take the underlying $A$-path and extend it to a $\Gamma(A)$-path $(\alpha(t),p_0)$ and then take the holonomy class.
Combining Proposition \ref{Proposition:A-homotopic-holonomic} and Lemma \ref{lemma:A:homotopypaths:Gamma(A)paths}, it follows that the holonomy class of $(\alpha(t), p_0)$ depends only on the $A$-homotopy class $[a(t)]$,  so the map $\Pi_1(A) \rightarrow \Hol(A)$ is well-defined. 

Finally, to go from an element of $\Hol(A)$ to an element of $\Hol(\mathcal{F})$, we take the underlying $A$-path $a(t)$ and apply the anchor map to get a path in the foliation $\rho(a(t))$.
Then we take the foliation holonomy class as per Garmendia and Villatoro~\cite{SingJoelAlf}.

All of these maps are surjective and are the identity on the units. Showing that they are homomorphisms only requires checking that the groupoid structure in $\GammaPath(A)$ is compatible with the groupoid structures in $\Pi_1(A)$, $\Hol(A)$, and $\Hol(\mathcal{F})$.
This holds almost by definition, as the groupoid structure on $\Pi_1(A)$ can be defined as the one induced by the flow product.
However, classically the groupoid structure on $\Pi_1(A)$ is defined by concatenation of $A$-paths (See Definition~\ref{definition:concatenation} for the precise definition of concatenation).

In Proposition~\ref{proposition:flow.product.compatible.with.homotopy}, we see that the flow product and concatenation of two elements of $\GammaPath(A)$ result in homotopic elements.
From this it follows that the two binary operations induce the same groupoid structure on $\Pi_1(A)$ and the map $\GammaPath(A) \to \Pi_1(A)$ is a groupoid homomorphism.
Concatenation is also used in Garmendia and Villatoro~\cite{SingJoelAlf} to define the groupoid structure on $\Hol(\mathcal{F})$ and a similar argument shows that this is compatible with the flow product. 
\end{proof}

\begin{example}
    The sequence in the above theorem recovers the expected objects in the basic examples. If $A= \mathfrak{g} \rightarrow \{\ast\},$ then \[\Hol(A) = \Inn(\mathfrak{g})\]
    by Example \ref{example:Lie-algebra}, while the holonomy groupoid of the characteristic foliation is trivial. If $A = M \times \mathfrak{g}$ has zero anchor, then the foliation has point leaves, but $\Hol(A)$ still records the inner automorphism data of $\mathfrak{g}$, as in Example \ref{free-transitive-lie algebroid-holonomy}. If $A= TM$, then by Example \ref{example:holonomy-algebroid-tangent-bundle} the holonomy groupoid is the pair groupoid on connected components. Finally, if $A= \F \subset TM$ is a regular foliation, then by Example \ref{example:algebroid-holonomy-foliations} the algebroid holonomy groupoid is the usual holonomy groupoid of $\F.$
\end{example}

The construction of $\Hol(A)$ was carried out entirely at the level of $\Gamma(A)$-paths and adjoint flows. When $A$ is integrable, however, one expects the holonomy to admit a description directly in terms of an integrating groupoid. The purpose of the next section is to develop such a description and to compare algebroid holonomy with the holonomy induced by local and global Lie groupoids.

\section{Holonomy of groupoids}\label{section:holonomy-of-groupoids}

Lie groupoids also have a natural notion of holonomy that is compatible with the one we have just defined for Lie algebroids.

In this section we will talk about these issues in terms of \emph{local} Lie groupoids. This adds some technical complications, unfortunately, but it will be necessary in the proof of longitudinal smoothness of the holonomy.

\subsection{Local Lie groupoids}

A local Lie groupoid is defined in essentially the same way as a Lie groupoid except for the fact that the multiplication and inverse operations may only be defined for elements ``close'' to the units. The primary differences are as follows:

\begin{itemize}
\item The multiplication function $m \colon \calG \times \calG \dashrightarrow \calG$ is a partial function and is well-defined only on an open neighborhood $\calM$ of the "axes": 
$$ (\calG \times_{s,t} 1_M) \cup  (1_M \times_{s,t} \calG) \subset \calM \subset \calG \times \calG. $$ 
Here $\mathcal{G}\times_{s,t}1_M:=\{(g,1_x): s(g) = x\}$ and $1_M \times_{s,t}\mathcal{G}=\{(1_x,g):t(g) = x\}$.
\item The inverse function $i \colon \calG \dashrightarrow \calG$ is a partial function and is well-defined only on an open neighborhood $\calI$ of the unit section $1_M$:
$$ 1_M \subset \calI \subset \calG. $$
\item The axioms of a groupoid are the same but are only assumed to hold on the domains where they are well-defined. This is most relevant for associativity, which is only assumed to hold on the intersection of the domains where the two different ways of multiplying three elements are well-defined.
\end{itemize}

Following are some basic facts about local Lie groupoids.

\begin{itemize}
\item Every local Lie groupoid defines a Lie algebroid. The definition is the same as with a traditional Lie groupoid:
 $$ A = \Lie(\calG) := T^s \calG|_{1_M}. $$
\item The source distribution $T^s \calG$ formed by the fibers of the source map is isomorphic to the pullback of the Lie algebroid $A$ along the target map:
$$  t^* A  \cong T^s \calG, \qquad (g,a) \mapsto dR_g (a). $$
\item This identification provides a bijection between sections of $A$ and right-invariant vector fields on $\calG$. In particular, the Lie bracket on $A$ is determined by the Lie bracket of the corresponding right-invariant vector fields.
$$ \Gamma(A) \to \mathfrak{X}^R(\calG), \qquad \alpha \mapsto \hat \alpha. $$
\item Every Lie algebroid is isomorphic to the Lie algebroid of a local groupoid.
\end{itemize}
\begin{terminology}
Given a Lie algebroid $A$ a \emph{local integration} of $A$ is a local Lie groupoid $\calG$ together with an algebroid isomorphism between $A$ and $\Lie(\calG)$. In general, we will omit the algebroid isomorphism and assume that $A = \Lie(\calG)$.

Given a local Lie groupoid $\widetilde{\calG}$,  we say $\calG \subset \widetilde{\calG}$ is an \emph{open restriction} if it is an open neighborhood of the units. We always think of open restrictions as inheriting a local groupoid structure from their ambient space.
\end{terminology}

In general, there are two main problems that arise when working with local groupoids:  First,  identities, which hold for general groupoids,  do not always hold for arbitrary local groupoids, and the second, the topology of the source-fibers of an arbitrary local groupoid can cause technical barriers.

\subsubsection{Local algebra principle}
In order to  deal with the first issue, we will appeal to what we call the local algebra principle. The motivation behind this discussion is the following observation: Equations that hold for arbitrary elements in arbitrary groupoids do not hold for arbitrary local groupoids. However, they will always hold \emph{within some open restriction}. 

For example, an equation such as:
\[ ((a \cdot b) \cdot c )\cdot a^{-1}  = (a \cdot ( b \cdot a^{-1})) \cdot ( a \cdot (c  \cdot a^{-1} ))    \]
may not hold for all elements due to issues with the domains of the groupoid identities needed to establish their equality. However, there will always exist an open restriction in which such an identity holds.

\begin{lemma}[Local algebra Principal]\label{lemma:local-algebra-principal}
    Given a local Lie groupoid $\widetilde{\calG}$, and any finite set of true algebraic identities in the language of groupoids, there will exist an open restriction $\calG$ in which this finite set of identities holds.
\end{lemma}

\subsubsection{Source-linear integrations}

We need the control over the topology of the local groupoids we are working with so that we can derive some geometric facts. To that end we will use the concept of what we call a source-linear integration.
\begin{lemma}
    Suppose $A$ is a Lie algebroid. There exists an open neighborhood of the zero section $\calG \subset A$ that can be equipped with a local groupoid structure with the following properties:
    \begin{enumerate}
    \item The source fibration is the vector bundle projection:
    \[ s = \pi_A|_{\calG}.\]
    \item The unit embedding is the zero section:
    \[ u  \colon M \to \calG, \qquad x \mapsto 0_x.\]
    \item The Lie algebroid $\Lie(\calG)$ is isomorphic to $A$ under the vertical lift map:
    \[ \ell \colon A \to \Lie(\calG) = T^s\calG|_{1_M},  \qquad  \ell(a) = \left. \frac{d}{dt}\right\vert_{t=0} t a. \]
    \item The exponential map for the isotropy Lie algebras is the identity map:
    \[ \exp_x = \Id \colon \mathfrak{g}_x \dashrightarrow \calG_x.\]
    \end{enumerate}
    Furthermore, every local integration of $A$ admits an open restriction isomorphic to a source-linear integration.
\end{lemma}
Although the terminology is novel, the concept itself is not. Indeed,  one of the main results of Cabrera-M\v{a}rcu\c{t}-Salazar \cite{Cabrera-Ioan-Maria-local-integration-Lie-brackets} is the construction of a local integration of a Lie algebroid on a neighborhood of the zero section using a Lie algebroid spray. Such constructions traditionally depend on a auxiliary choice, such as a spray or a connection. Since the specific choice plays no role in our arguments, we only use the existence of source-linear integrations. 

\subsubsection{Bisections}

To study holonomy from the perspective of local groupoids, it is necessary to work with bisections, which play the role of global counterparts of sections of Lie algebroid. 

The following definition is standard in the theory of  Lie groupoids. 
\begin{definition}
    Suppose $\calG$ is a local Lie groupoid. A \emph{bisection} of $\calG$ is a smooth map \[\sigma \colon M \to \calG,\] such that
    \begin{enumerate}
        \item $s \circ \sigma = \Id_M,$ and 
        \item $t \circ \sigma: M \rightarrow M$ is a diffeomorphism.
    \end{enumerate}
    The \emph{support} of a bisection is the subset
    $$supp(\sigma):=\{ x \in M : \sigma(x) \notin  1_M \} \subset M.$$

    A bisection $\sigma$ is said to be compactly supported if 
 $\overline{supp(\sigma)}$
is compact. We denote by   \[\Bis_c(\calG)\]  the set of all compactly supported bisections of $\calG$.
\end{definition}
Bisections admit a natural multiplication. Given $\sigma, \eta \in   Bis(\mathcal{G})$, their product is defined by: 
$$   (\sigma \cdot \eta)(x) := \sigma(t(\eta(x)) \cdot \eta(x).$$

Bisections may be viewed as a global counterpart of sections of the Lie algebroid. Indeed, there is a natural integration procedure that associates a one-parameter family of bisections to a time-dependent section of the algebroid. This perspective motivates thinking of $Bis_c(\mathcal{G})$ as a local group integrating the Lie algebra $\Gamma_c(A).$

This following terminology is nonstandard, but it will be  linguistically useful for our purposes.

\begin{definition}\label{definition:G-integrable-section}
    Suppose $\alpha \in \Gamma^t_c(A)$ is a time-dependent section of $A$. 
    Let $\hat{\alpha} \in \mathfrak{X}^R(\calG)$ be the associated right-invariant vector field on $\calG$.
    Define a one-parameter family of bisections by
    \[ \sigma(t)(x) := \Phi^t_{\hat{\alpha}} (1_x). \]
    If $\sigma(t)$ is defined for all $t \in [0,1]$,  then we say that $\alpha$ is \emph{$\calG$-integrable}. In this case, the family   $\sigma(t)$ is  called the  \emph{$\calG$-integration} of $\alpha$.
\end{definition}
For a Lie groupoid (not local), the above construction yields a bijective  correspondence between time-dependent sections of the Lie algebroid and one-parameter families of bisections starting at the identity. 

For local groupoids, however, this correspondence is more subtle. In particular, existence and uniqueness may depend on the precise class of local groupoids under consideration and on the domains on which the relevant structure maps are defined.


\begin{definition}
    A local integration $\calG$ of $A$ is said to have \emph{differentiable bisections} if, for  every smooth one-parameter family of   of bisections \[ \sigma(t) \in Bis(\mathcal{G}), \quad \quad \sigma(0) = 1_M,\] there exists a  unique time-dependent section $\alpha \in \Gamma^t_c(A)$ whose $\mathcal{G}$-integration is $\sigma(t)$.
\end{definition}

\begin{lemma}
   Let $\widetilde{\calG}$ be a local integration of $A$. Then there exists an open restriction $\calG \subset \widetilde{\calG}$ such that $\mathcal{G}$ has differentiable bisections.
\end{lemma}

\begin{proof}
Let $\tilde{\G}$ be a local integration of $A$, and let $\sigma(t) \in Bis(\tilde{\G})$ be a smooth one-parameter family of bisections satisfying $\sigma(0) = 1_M.$

Choose an open neighborhood $U \subset \tilde{\G}$ of the unit section such that $g^{-1}\cdot g$ is defined and equal to $1_{s(g)}$ for every $g \in U.$ If $\sigma(t)$ takes values in $U$, then the associated time-dependent section of $A$ is given by \[ \alpha(t):=dR_{\sigma(t)^{-1}}(\sigma^\prime(t)).\]
Indeed, the above expression is well-defined because the right-translation by $\sigma(t)^{-1}$ identifies $T^s_{\sigma(t)} \tilde{\G}$ with $A.$

Let $\mathcal{G} \subset \tilde{\G}$ be an open restriction contained in $U.$ Although the expression defining $\alpha(t)$ may involve elements outside $\mathcal{G}$, it remains well-defined using the ambient groupoid structure of $\mathcal{G}$.

For every $g \in U$, the identity \[ g \cdot (g^{-1} \cdot g) = g\]

implies that \[ dR_{\sigma(t)} \circ dR_{\sigma(t)^{-1}}  = Id_{T^s_{\sigma(t)}\tilde{\G}}.\]

Hence, 

 \[\hat{\alpha}(t)|_{\sigma(t)} =  dR_{\sigma(t)} \circ dR_{\sigma(t)^{-1}} (\sigma'(t)) = \sigma'(t). \]
 It follows that $\sigma(t)$ is the $\mathcal{G}$-integration of $\alpha(t).$ Therefore every smooth one-parameter family of bisections starting at the identity arises from a unique time-dependent section of $A,$ and $\mathcal{G}$ has differentiable bisections.

\end{proof}

\subsubsection{Convenient integrations}
Motivated by the preceding discussion, we introduce a class of local integrations that satisfy the technical properties required in the arguments below.

\begin{definition}
 Suppose $A$ is a Lie algebroid. A local integration $\mathcal{G}$ of $A$ is said to be  \emph{convenient} if the following conditions hold: 
    \begin{enumerate}
        \item $\mathcal{G}$ is source-linear;
        \item the source fibers of $\G$ are convex;
        \item $\G$ has differentiable bisections;
        \item every universally valid groupoid identity involving fewer than one hundred elements holds in $\G$  when well-defined. 
    \end{enumerate}
\end{definition}

\begin{remark}
\begin{enumerate}
    \item  The particular bound of one hundred is arbitrary. None of the arguments below require identities involving anywhere near this many elements. The purpose of the definition is simply to ensure that, throughout our constructions, all groupoid identities that arise are automatically valid in the chosen local integration.

   \item  The finite bound in the definition of a convenient integration is essential. Indeed, Fernandes-Michiels \cite{Fernandes-Daan-Associativity-Integrability} proved that the failure of integrability of a Lie algebroid is reflected in the failure of associativity of any local integration. In particular, a local Lie groupoid is globally associative if and only if its Lie algebroid is integrable. Consequently, for a nonintegrable Lie algebroid one cannot require arbitrary groupoid identities to hold in a local integration.
    \end{enumerate}
\end{remark}

Convenient integrations are general in the sense that all integrations admit an open restriction which is isomorphic to a convenient integration.

\subsection{Holonomy of a groupoid}

We shall show that holonomy can be described entirely in terms of an integrating groupoid. The resulting characterization is the groupoid analogue of the classical Lie group fact that the adjoint flow of a time-dependent Lie algebra element is induced by conjugation by the corresponding integrated one-parameter family of group elements obtained through the Lie group exponential map. 

The key observation is that every bisection of a local Lie groupoid $\G$ determines a conjugation action on $\G$. This action will serve as the basic mechanism for describing holonomy.

\begin{definition}
Given a bisection $\sigma$ of a local Lie groupoid $\calG$, we say that $\sigma$ is \emph{invertible} if there exists a bisection $\sigma^{-1}$ satisfying \[ \sigma \cdot \sigma^{-1} = \sigma^{-1} \cdot \sigma = id_M,\]

wherever these products are defined.

Given an invertible bisection $\sigma,$ the \emph{conjugation action} is the local groupoid morphism
$$ C_\sigma \colon \calG \dashrightarrow \calG, $$
defined by the formula:
$$ C_\sigma(g) := (\sigma(t(g)) \cdot g ) \cdot \sigma(s(g))^{-1}. $$
\end{definition}
A standard argument shows that the conjugation action is a groupoid homomorphism for any Lie groupoid and for any two bisections $\sigma$ and $\eta$ of $\mathcal{G}$, one obtains
\[ C_{\sigma} \circ C_\eta = C_{\sigma \cdot \eta}. \]

For a local Lie groupoid, the local algebra principle guarantees the existence of an  open neighborhood of the units in $\calG$ on which the preceding properties are valid. In particular, the last identity will hold for any convenient local integration.

The fact that the conjugation action is a local groupoid morphism implies that its infinitesimalization defines a Lie algebroid automorphism of the associated Lie algebroid $A.$

\begin{definition}
    Given a bisection $\sigma$ of a local integration $\calG$ of $A$, then the associated \emph{adjoint action} of $\sigma$ on $A$ is bundle map
$$ \Ad_\sigma \colon A \to A $$
defined by the formula:
$$ \Ad_\sigma(a) := dC_\sigma (a). $$
\end{definition}
Similar to the previous comments about the conjugation action, the adjoint action will be a Lie algebra morphism for any Lie groupoid. For a local Lie groupoid one may, in principal, need to restrict ones view to bisections which take values in a neighborhood of the identity.

Using the adjoint action, we can now define the holonomy of a (local) Lie groupoid.

\begin{definition}
    Suppose $\calG$ is a local groupoid. Given $g \in \calG$ the \emph{holonomy} $\hol_\calG(g)$ (reduced) holonomy class of the full holonomy transformation:
    $$ \Ad_\sigma  |_{A_{S_x}}  \colon A_{S_x} \to A_{f(S_x)} $$
    where $S_x$ is a slice through $x = s(g)$ and $f := t \circ \sigma$ is the diffeomorphism associated to $\sigma$.
\\

    We say that $\calG$ \emph{has holonomy} if \begin{enumerate}
        \item $\Ad_\sigma$ is a Lie algebroid morphism for any bisection $\sigma$, and 
   \item  for all $g \in \calG$, the holonomy class $\hol_\calG(g)$ does not depend on the choice of bisection.
    \end{enumerate}
    
    In such a case, we will write $\Hol(\calG) \subset \HT$ to denote the image of $\calG$ under $\hol_\calG$.
\end{definition}

The last definition is one that is only relevant for local Lie groupoids. Every classical Lie groupoid has holonomy(see Lemma \ref{lemma:Lie-grpd-has-holonomy}). For Local Lie groupoids, this is still not such a restrictive condition since one can always ``shrink'' a local Lie groupoid to one that has holonomy.

\begin{lemma}\label{Lemma:open-restriction-has-holonomy}
    Suppose $\widetilde{\calG}$ is a local Lie groupoid with associated Lie algebroid $A$. Then there is an open restriction $\calG \subset \widetilde{\calG}$ of that has holonomy.
\end{lemma}
\begin{proof}
Without loss of generality, we assume that we are working with a convenient integration of $\widetilde{\calG} \subset A$.

Suppose $\sigma_1$ and $\sigma_2$ are two bisections through $g \in \calG$. We must show that these two bisections define the same holonomy transformation. Since holonomy transformations only depend on the algebroid morphism induced on a slice, by restricting the groupoid to the said slices, we can assume without loss of generality that the orbit associated to $g$ is zero dimensional.

Consider the bisection $\epsilon := \sigma_1^{-1} \cdot \sigma_2$, when defined,  satisfies \[ \epsilon(x) = 1_x,\] that is, it is a bisection through the identity defined on a neighborhood of $x$.  $\mathcal{G}$ is a source-linear integration, so locally $\mathcal{G}$ can be identified with an open neighborhood of the algebroid $A$.  In this correspondence, $1_x$ is just the zero vector $0_x$ and thus $\epsilon$ vanishes at $x$. 
Because $\mathcal{G}$ is convex,  for each $\epsilon(y)$ the ray connecting $0_y$ and $\epsilon(y)$ is in $\mathcal{G}$. Furthermore, since the orbit of $x$ is zero dimensional, there will be a neighborhood of $x$ where for all scalars $t \in [0,1]$ we have that $t \cdot \epsilon$ is a bisection.

We consider $t \cdot \epsilon$ to be a time-dependent bisection. Since $\calG$ has differentiable bisections, there exists a unique time-dependent bisection $\alpha \in \Gamma^t_c(A)$ such that $t \cdot \epsilon$ is the $\calG$-integration of $\alpha$. Furthermore, such an $\alpha$ must vanish at $x$ since $t \cdot \epsilon(x)$ is constant. 

From this, we conclude that the holonomy transformation defined by $\Ad_{\epsilon}|_{S_x}$ is trivial since it is the time-1 flow of a time-dependent section in $I_x \Gamma^t_c(A)$.

Since, $\calG$ admits many algebraic identities, the following computation holds whenever it is well-defined:
$$ (C_{\sigma_1})^{-1} C_{\sigma_2} =C_{\sigma_1^{-1}} C_{\sigma_2} = C_{\epsilon}. $$
From this,  we conclude
$$ \Ad_{\sigma_1}^{-1} \Ad_{\sigma_2} = \Ad_{\epsilon}.$$ In other words,
$$ \Ad_{\sigma_2} = \Ad_{\sigma_1} \circ \Ad_{\epsilon}.  $$
Since $\Ad_{\epsilon}$ defines a trivial holonomy transformation, we conclude that $\sigma_1$ and $\sigma_2$ define the same holonomy transformation. This argument does rely on the well-definedness of the above expressions, however.

To that end, choose $\calG$ to be an open restriction of $\widetilde{\calG}$ such that all of the above expressions are well defined (at least inside of $\widetilde{\calG}$). Then $\calG$ has holonomy.
\end{proof}

\subsection{Relation to Algebroid holonomy}

The notion of groupoid holonomy is compatible with algebroid holonomy and they are, in fact, two different ways of looking at the same phenomena.

\begin{lemma}\label{Lemma:Local-Integration-Adjoing-flow-adjoint-representation}
    Suppose $\widetilde{\calG}$ is a local integration of a Lie algebroid $A$. There is an open restriction $\calG \subset \widetilde{\calG}$ with the following property: For all time-dependent sections $\alpha \in \Gamma^t_c(A)$ with $\calG$-integration $\sigma(t)$ we have that:
    \[ \Phi^t_{\alpha} = \Ad_{\sigma(t)}. \]
    In other words, the adjoint flow and adjoint representation coincide.
\end{lemma}
\begin{proof}
This follows from the fact that the Lie bracket on sections of $A$ can be characterized as the derivative of the adjoint action. More specifically, given a time-dependent section $\sigma(t)$ with $\sigma(0) = 1_M$ and $\sigma'(0) = \alpha$ we have that:
\[ \left. \frac{d}{dt}\right|_{t=0} \Ad_{\sigma(t)} \beta = \left[ \beta, \alpha \right].\]

\end{proof}
An immediate corollary of this fact is that, in a suitably small neighborhood of the units, algebroid holonomy and local groupoid holonomy coincide.

\begin{corollary}\label{corollary:open-neighborhood-identity-holonomy}
    Suppose $\calG$ is a local integration of a Lie algebroid $A$ and $\calG$ has holonomy. Then there is an open neighborhood $U \subset \calG$ of the identity with the following property: For all time-dependent sections $\alpha \in \Gamma^t_c(A)$ with $\calG$-integration $\sigma(t)$ contained in $U$, we have that: 
    $$ \forall x \in M, \qquad \hol^\calG(\sigma(1)(x)) = \hol(\alpha,x).$$
\end{corollary}

When $\mathcal{G}$ is a global Lie groupoid, the locality issues that arise for local integrations disappear. In particular, one does not need to shrink to sufficiently small neighborhoods of the unit section to ensure that products, inverses, and bisections are defined. Therefore, for an integrable Lie algebroid $A,$ the holonomy transformation  can be computed using any source-connected integration of $A.$

\begin{corollary}\label{Cor-Groupoid-Hol-Equal-Algeb-Hol}
    Suppose $A$ is an integrable Lie algebroid and $\calG$ is a source connected Lie groupoid (not local)  that integrates $A$. Then $\Hol(\calG) = \Hol(A)$.
\end{corollary}
For local integrations, an analogous statement holds. As usual, one must first restrict to a sufficiently small neighborhood of the unit section in order to ensure that all relevant structure maps are defined. 
\begin{proposition}\label{proposition:open-groupoid-morphism-proposition} 
    Suppose $A$ is a Lie algebroid and let $\Pi_1(A)$ be the Weinstein groupoid of $\calG$ and $A$ be an arbitrary local integration. There exists an open restriction $\calG \subset \widetilde{\calG}$ with holonomy, together with an open groupoid homomorphism $\calG \to \Pi_1(A)$ which makes the following diagram commute:
    \[
    \begin{tikzcd}
      \calG \arrow[r, "\hol_\calG"] \arrow[d]& \Hol(\calG) \arrow[d, hook] \\
  \Pi_1(A) \arrow[r] & \Hol(A) 
    \end{tikzcd}
    \]
    where the right vertical arrow is the natural inclusion of $\Hol(\calG)$ into $\Hol(A)$ as subsets of $\HT(A)$.
\end{proposition}
\begin{proof}
    To prove this proposition, we will need to refer to the construction of the Weinstein groupoid from Crainic and Fernandes~\cite{IntegrabilityofLiebrackets}. It will be useful to summarize it briefly:

    Given a Lie algebroid $A$ with a connection $\nabla$ they construct a natural exponential map $\exp^\nabla \colon A \dashrightarrow \mathcal{P}(A)$ which is defined in a neighborhood of the zero section.
    Crucially, they show that the $A$-homotopy equivalence relation induces an infinite-dimensional foliation on $\mathcal{P}(A)$\footnote{Technically they put some finite differentiability conditions on the space of $A$-paths to ensure it is a Banach manifold and have a proper foliation theory with a Frobenius theorem.}.
    Furthermore, they show that the image of the exponential $\exp^\nabla$ is transverse to the foliation.
    The Weinstein groupoid is the leaf space of the foliation on $\mathcal{P}(A)$ and composing $\exp^\nabla$ with the natural quotient map yields an open map:
    $$ \pi \colon A \dashrightarrow \Pi_1(A) $$
    Crainic and Fernandes also showed that there is an open neighborhood of the zero section of $A$ which inherits a unique local groupoid structure compatible with the exponential map and the quotient map to $\Pi_1(A)$.
    Cabrera, M\v{a}rcu\c{t}, and Salazar~\cite{Cabrera-Ioan-Maria-local-integration-Lie-brackets} showed that the germ of this local groupoid structure is unique and independent of choice of connection.

    It suffices to show that the relevant diagram exists for any convenient integration of $A$. Since the map $\Hol(\calG) \to \Hol(A)$ is supposed to be the inclusion of a subset, what we must actually show is that the following triangle commutes:
    $$
    \begin{tikzcd}
        \calG \arrow[dr, "\hol_\calG"] \arrow[dd, "\pi"]& \\
        & \HT(A) \\
        \Pi_1(A) \arrow[ur, "\hol"'] &
    \end{tikzcd}    
    $$
    Furthermore, following from the preceding discussion, we can assume that $\calG \subset A$ is an integration of $A$ arising from a connection $\nabla$ on $A$ and we have the associated open local groupoid homomorphism $\pi \colon \calG \to \Pi_1(A)$. This forms the left vertical arrow of our diagram.

    To see why this diagram commutes: Suppose we are given $g \in \calG$ arbitrary and let $x$ be the source of $g$. There is an associated $A$-path $a(t) = \exp^\nabla(g)$. Furthermore, given any time-dependent section $\alpha(t)$ extending $a(t)$ the associated $\calG$-integration $\sigma(t)$ will satisfy $\sigma(1)(x) = g$. 
    By Corollary~\ref{corollary:open-neighborhood-identity-holonomy}, we have that $\hol_\calG(g) = \hol(\alpha, x)$ for $x = s(g)$. This shows that the holonomy transformation defined by $g$ in $\calG$ coincides with the holonomy transformation defined by the corresponding $A$-path, and hence $\hol_\calG(g) = \hol(\pi(g))$.
\end{proof}

\section{Examples of Algebroid Holonomy}\label{section:examples}

\begin{enumerate}

\item \label{example:Lie-algebra} (Lie algebras) Let $A=\mathfrak{g} \rightarrow \{\ast\}$ be a finite-dimensional Lie algebra, and  let $G$ be a simply-connected integration of $\mathfrak{g}$. 

Since the basis consists of a single  point, every slice is the point itself and the the full holonomy transformation groupoid and reduced holonomy transformation groupoids coincide: 
\[ \FHT = \HT  = \mathrm{Aut}(\mathfrak{g},\mathfrak{g}). \]

By Corollary \ref{Cor-Groupoid-Hol-Equal-Algeb-Hol}, the algebroid holonomy may be computed using any source-connected integration. Thus the holonomy of a $\Gamma(A)$-path is determined by the adjoint action of its endpoint in $G$. More precisely, if $v(t) \in \mathfrak{g}$ and $g_v(t) \in G$ denotes the integrated path satisfying \[ g^\prime_v(t) = v^R(t)_{g_v(t)}, \quad \quad g_v(0) = e,\]
then, by Lemma \ref{Lemma:Local-Integration-Adjoing-flow-adjoint-representation}, \[ \Phi^t_v = Ad_{g_v(t)}.\]
Consequently, two $\Gamma(A)$-paths are holonomically equivalent precisely when they determine the same inner automorphism of $\mathfrak{g}$. Therefore, \[ \Hol(\mathfrak{g}) \cong \Inn(\mathfrak{g}) = \Ad(G).\]
Thus the algebroid holonomy groupoid recovers the classical adjoint representation of a Lie algebra.

\item \label{example:integrable-Lie-algebroid-holon}(Integrable Lie algebroids) Let $A$ be an  integrable Lie algebroid and let $\G \rightrightarrows M$ be a source-connected integration.

By Corollary \ref{Cor-Groupoid-Hol-Equal-Algeb-Hol}, one can use groupoid holonomy to compute the holonomy of $A$.

Recall that given a local bisection $\sigma \colon U \to \calG$ of $\calG$,  one can define a conjugation map $\Ad_\sigma \colon A|_U \dashrightarrow A|_U$ defined on the domain of $\sigma$. Groupoid holonomy provides us with a homomorphism
\[ \hol_{\G}: \G \rightarrow \Hol(A).\]
For $g \in \calG$, $\hol_{\calG}(g)$ is defined to be the holonomy transformation class of $\Ad_\sigma$, where $\sigma$ is any local bisection extending $g$.

Indeed, it turns out that if the holonomy transformation class of $g$ is trivial then there exists a local bisection $\sigma \colon U \to \calG$ extending $g$ where $\Ad_g = \Id_{A|_U}$. Hence, computing the holonomy of an integrable algebroid reduces to determining which elements of the integration can be extended to a bisection which acts ``trivially'' on the underlying algebroid.

\item  (The trivial Lie algebroid) Let \[ A = M \times \{0\} \rightarrow M\]
  be the trivial Lie algebroid. 
  Since the characteristic foliation consists of points, every slice through $x$ is a neighborhood of $x$. The associated slice algebroid is again the trivial rank-zero algebroid. Consequently, \[ \FHT(A) = \bigsqcup_{x \in M} \bigsqcup_{S_x,  S^\prime_x} Diff({S_x}, S_{x}^\prime) .\]

  Every time-dependent section of $A$ is identically zero and therefore every adjoint flow is trivial. 
 Therefore, 
\[ \THT(A) = \bigsqcup_{x \in M} \bigsqcup_{x \in S_x} {Id_{S_x}}. \]
 Consequently, \[\HT(A) = \FHT(A).\]
 
 On the other hand, every $\Gamma(A)$-path is necessarily constant. Therefore, there is exactly one holonomy class over each point and \[ \Hol(A) = 1_M \rightrightarrows M.\]

\item \label{free-transitive-lie algebroid-holonomy} 
(A trivial bundle of Lie algebras) 
Let \[A=M \times \mathfrak{g}\rightarrow M\] be the trivial Lie algebroid with  anchor $\rho\equiv 0$. Since the anchor vanishes,  the characteristic foliation consists of  points. Consequently, for a slice $S_x$ through $x \in M$, we have \[ A_{S_x} = S_x \times \mathfrak{g}.\]
Hence a holonomy transformation from $S_x$ to $S_y$ is simply a germ of a Lie algebroid isomorphism \[ S_x \times \mathfrak{g} \rightarrow S_y \times \mathfrak{g}.\]
Therefore,
  \[\FHT(A)= \bigsqcup_{x \in M} \bigsqcup_{S_x}Aut(S_x \times \mathfrak{g}) = \bigsqcup_{x \in M} \bigsqcup_{S_x} \mathrm{Diff}_x(S_x) \times C^\infty(S_x, Aut(\mathfrak{g})). \]
If $\alpha(t)$ is a time-dependent section of $A$, then its adjoint flow acts pointwise on the $\mathfrak{g}$-factor. By the Lie algebra computation of Example \ref{example:Lie-algebra}, the resulting automorphisms are inner automorphisms of $\mathfrak{g}$. Hence, 

\[ \THT_x = \{ (\Id, x \mapsto L_x) : \forall y \in S_x , L_y \in \Inn(\mathfrak{g})  \text{ and } L_x = \Id \}. \]

Moreover, two elements  $\theta = (f, L)$ and $\theta^\prime = (f^\prime, L^\prime) \in \HT(A)$ are in the same reduced holonomy class if and only if  $f|_{S_x} = f^\prime|_{S_x}$, $L_x = L^\prime_{x}$,  and $L^{-1}_y \circ L^\prime_y \in \Inn(\mathfrak{g})$, for all  $y \in S_x$. 
Consequently, every $\Gamma(A)$-path remains at its base point.

To compute the holonomy groupoid, we observe that $A$ will be an integrable algebroid. Its source simply connected integration will be given by the trivial bundle of Lie groups $ M \times G \grpd M$,  where $G$ is the simply connected integration of $\mathfrak{g}$.

Since $\Hol(A) = \Hol(M \times G)$,  it suffices to compute which elements of $ M \times G$ are associated to trivial holonomy transformations.

Given $(x, g) \in M \times G$ if $\Ad(g) \neq \Id_{\mathfrak{g}}$ then any bisection $\sigma$ extending $g$ will not induce a trivial holonomy transformation.
On the other hand, if $\Ad(g) = \Id_{\mathfrak{g}}$ then we can extend $(x, g) $ to a constant bisection $\sigma$ and observe that $\Ad_\sigma = \Id_A$. Therefore, $(x, g) $ has trivial holonomy if and only if $\Ad_g = \Id_{\mathfrak{g}}$. From this we conclude that:
$$ \Hol(M \times \mathfrak{g}) = \Hol(M \times G) = M \times \Inn(\mathfrak{g}).$$

\item \label{example:holonomy-algebroid-tangent-bundle} (The tangent algebroid) Let $A = TM$ be the tangent bundle of $M$, with $\rho \equiv Id_{TM}$.  The characteristic foliation consists of the connected components of $M$. For every $x \in M$, a slice through $x$ is singleton. Hence the slice algebroid is the trivial Lie algebroid over a point.  Restricting the algebroid to the slice results in a zero algebroid over a point. Between any two points there is a single algebroid morphism relating trivial algebroids over them. Therefore,  
\[\FHT(A)= \HT(A) =  \bigsqcup_{x, y \in M} \Alg(\{ 0_x \} , \{ 0_y \} ) \cong M \times M.\]

Two $\Gamma(A)$-paths $(\alpha(t), p_0)$ and $(\beta(t), p_0)$ between the same two points are always holonomic because their time-1 flows are trivially equal when restricted to $\{0_{p_0} \}$. Therefore, the holonomy groupoid is \[\Hol(A) = \{ (x,y) \in M \times M \ : \ x \text{ and } y \text{ are in the same connected component}\}. \]

Thus the holonomy construction recovers the pair groupoid.

For an arrow $g:x\to y$ of $G$, conjugation induces a Lie algebra
isomorphism
\[
\Ad(g):\mathfrak g_x\longrightarrow \mathfrak g_y.
\]

Two arrows of $G$ determine the same element of $\Hol(A)$ precisely
when they induce the same holonomy class in
\[
\HT(A).
\]

Consequently, $\Hol(A)$ is the groupoid whose arrows from $x$ to $y$
are the holonomy classes of Lie algebra isomorphisms
\[
\mathfrak g_x\longrightarrow \mathfrak g_y
\]
induced by arrows of a source-connected integration of $A$.
\item \label{example:algebroid-holonomy-transitive} (Transitive Algebroid) Let $A \rightarrow M$ be a transitive algebroid. Since the anchor is surjective, the characteristic foliation consists of a single leaf, namely $M$ itself. Consequently, a slice through $x \in M$ is the point $\{x\}$ and the corresponding slice algebroid is the isotropy Lie algebra \[ \mathfrak{g}_x:=ker(\rho_x).\]
Therefore, \[\FHT(A) = \bigsqcup_{x, y \in M} Alg(\mathfrak{g}_x, \mathfrak{g}_y).\]  
If $\alpha(t) \in I_s \Gamma(A)$, then its time-1 flow restricts to a Lie algebra automorphism of $\mathfrak{g}_x.$ Hence

\[\THT(A)=\bigsqcup_{x \in M} \{Id_{\mathfrak{g}_x}\}. \]

Consequently,  \[\HT(A) = \bigsqcup_{x, y \in M} \Alg(\mathfrak{g}_x, \mathfrak{g}_y).\]

Thus the holonomy transformation groupoid records isotropy Lie algebra isomorphisms modulo

Now, any two $\Gamma(A)$-path $(\alpha(t), p_0)$ and $(\beta(t), p_0)$ has the same holonomy class if they have the same automorphism from $\mathfrak{g}_x$ to $\mathfrak{g}_y$.

Suppose now that $A$ is integrable and let \[ \G \rightrightarrows M\] be the source-connected integration. By example \ref{example:integrable-Lie-algebroid-holon} and \ref{Cor-Groupoid-Hol-Equal-Algeb-Hol},  the holonomy of a $\Gamma(A)$-path depends only on the corresponding arrow of $\G.$ Furthermore, the discussion above shows that the resulting holonomy is completely determined by the induced isomorphisms between isotropy Lie algebras.

Consequently, \[ \Hol(A) \cong Im(Ad(\mathcal{G})) \subset \bigsqcup_{x,y \in M} \Alg(\mathfrak{g}_x, \mathfrak{g}_y),\]
where \[ \Ad(g): \mathfrak{g}_x \rightarrow\mathfrak{g}_y\]

is the isotropy adjoint representation.

\item \label{example:algebroid-holonomy-foliations} Let $\mathcal{F}$ be a foliation on $M$ and regard $A=  T\mathcal{F} \rightarrow M$ as a Lie algebroid with anchor map $i:T\mathcal{F} \hookrightarrow TM $. For any $x \in M$, let $S_x$ be a slice transverse to the foliation through $x$.  The corresponding slice algebroid is $A_{S_x} = S_x\times \{0\}$. Therefore, \[\FHT(A)=\bigsqcup_{x, y \in M} \bigsqcup_{S_x, S_y} Diff(S_x, S_y).\]

For any $\alpha(t) \in I_x \Gamma(A)$, the time-1 flow restriction to $S_x$ is
 restriction of the flow of $\rho(\alpha(t)) = \alpha(t) \in I_x \mathcal{F}$ to $S_x$. Thus, \[\THT(A)= \bigsqcup_{x \in M} \bigsqcup_{S_x}\{ \Theta \in \FHT(A): \Theta = \Phi^1_{\alpha(t)}, \alpha(t) \in I_x\mathcal{F}\}. \]
 and hence, \[\HT(A) = \bigsqcup_{x, y \in M} \bigsqcup_{S_x, S_y} Diff(S_x, S_y)\Big/\bigsqcup_{x \in M} \bigsqcup_{S_x}\{ \Theta \in \FHT(A): \Theta = \Phi^1_{\alpha(t)}, \alpha(t) \in I_x\mathcal{F}\}.\]

Note that there is an existing construction \cite{SingJoelAlf} for defining the holonomy groupoid of a foliation, this $\Hol(A)$ is exactly the holonomy of the foliation $\mathcal{F}$.

\item Let $$M = ([0,2\pi] \times \mathbb{R})/\sim \to [0,2\pi]$$ be a line bundle over a circle with one twist and consider the action of $\mathbb{R}$ on this by rotating around the circle. The corresponding action algebroid is  \[ A= M \times \mathbb{R} \rightarrow M\] with anchor  \[\rho: M \times \mathbb{R} \rightarrow TM, \quad \quad \rho(m,c):= cX(e),\]

where $X$ is the vector field in the $[0,1]$-direction(generating the circle action).

For $x = (\theta_0, y_0) \in M$, a slice at $x$ is 
 is just a small neighborhood  $S_x=\{(\theta_0, y):|y-y_0|<\epsilon\}$ and therefore, the slice algebroid is  $A_{S_x} = S_x \times \{0\}$.

The full holonomy transformation groupoid is \[\FHT(A)= \Big(\bigsqcup_{x \neq y\in M} \bigsqcup_{x \in S_x, y \in S_y} \{Id\} \Big) \hspace{0.5cm} \bigsqcup \hspace{0.5cm}\Big( \bigsqcup_{x \in M} \bigsqcup_{x \in S_x} \{Id, -Id\}  \Big).\]
Let $\alpha(t) \in I_x \Gamma(A)$,  and let $\phi^t_\alpha$ represent the flow of $\rho(\alpha(t))$. Then,   $\phi^t_{\rho(\alpha)}(x) =x$ for all $t$. Therefore  \[\THT(A)= \bigsqcup_{x \in M} \bigsqcup_{S_x}\{Id\}. \] Hence\[\hspace{1cm} \HT(A)=\FHT(A).\]


\textbf{Holonomy Groupoid: }
Two $\Gamma(A)$-paths $(\alpha(t), p_0)$ and $(\beta(t), p_0)$  have the same algebroid holonomy if $\rho(\alpha) $ and $\rho(\beta)$ have the same foliation holonomy. 
The holonomy groupoid is \[\Hol(A)=  \Hol(\mathcal{F}).\]

\item (Constant-rank Lie algebra actions) Let $\rho: \mathfrak{g} \rightarrow \mathfrak{X}(M)$ be a Lie algebra action of constant rank and let $A = \mathfrak{g} \times M$ be the associated action Lie algebroid. The anchor is given by $\rho_x(v) = \rho(v)(x).$ Let $\mathcal{F}:=\rho(\mathfrak{g}) \subseteq TM$ be the corresponding regular foliation. Since the rank of $\rho$ is constant, the isotropy Lie algebras $\mathfrak{g}_x:=ker(\rho_x)$ have locally constant dimension.  

For every $x \in M$, let $S_x$ be a slice transverse to the foliation $\mathcal{F}$. Since $A_{S_x} = \rho^{-1}(TS_x),$ and $TS_x \cap \rho(\mathfrak{g})  = 0,$ we obtain \[ A_{S_x} = \bigsqcup_{ y \in S_x} \mathfrak{g}_y.\]

Thus the slice algebroid is a bundle of isotropy Lie algebras over $S_x.$

Consequently, \[ \FHT(A) = \bigsqcup_{x, y \in M} \bigsqcup_{S_x, S_y} Alg(A_{S_x}, A_{S_y}).\]

Let $G$ be the simply connected Lie group integrating $\mathfrak{g}$. The action algebroid $A$ is integrated by the action groupoid \[ G \ltimes M \rightrightarrows M.\]

Since $G$ is connected, this groupoid is source-connected. Therefore, by Corollary \ref{Cor-Groupoid-Hol-Equal-Algeb-Hol} and Example \ref{example:integrable-Lie-algebroid-holon}, the holonomy of a $\Gamma(A)$-path may be computed entirely from the action groupoid.

An arrow $(g,x): x \rightarrow gx$ induces both \begin{enumerate}
    \item the holonomy transformation of the foliation $\mathcal{F}$, and 
    \item an isomorphism of isotropy Lie algebras $\Ad_g:\mathfrak{g}_x \rightarrow \mathfrak{g}_{gx}.$
\end{enumerate}

Hence the holonomy groupoid is the image of the map \[ G \ltimes M \rightarrow \Hol(\mathcal{F}) \times \bigsqcup_{x, y \in M} Alg(\mathfrak{g}_x, \mathfrak{g}_y)\]

which sends $(g,x)$ to the pair consisting of its foliation holonomy and the induced isotropy Lie algebra isomorphism.

In particular, this example interpolates between the case for $\rho \equiv 0$ (Example  \ref{free-transitive-lie algebroid-holonomy}) and for  $\mathfrak{g}_x = \{0\}$ (Example \ref{example:algebroid-holonomy-foliations}). Thus the holonomy groupoid of a constant-rank Lie algebra action simultaneously records the transverse holonomy of the foliation $\mathcal{F}$ and the adjoint representation of the isotropy Lie algebras.

\item \label{example:algebroid-holonomy-regular-poisson}(Regular Poisson Manifolds) Let $(M, \pi)$ be a regular Poisson manifold, and let \[ A = T^\ast_\pi M\]
denote its cotangent Lie algebroid. The anchor map is \[ \rho=\pi^\sharp: T^\ast M \rightarrow TM.\]
Since $\pi $ has constant rank, the image $\pi^\sharp $ defines a regular distribution \[ \mathcal{F}:=\pi^\sharp(T^\ast M) \subset TM,\]
which is precisely the symplectic foliation of $(M, \pi).$
Fix a point $x \in M,$  and let $S_x \subset M$ be a slice through $x, $ transverse to the symplectic leaf $L_x$ through $x.$ After shrinking $S_x$ if necessary, we may assume that \[ T_yM = T_yS_x \oplus T_y\mathcal{F} \quad \text{ for all } y \in S_x.\]
The slice algebroid is, by definition, \[A_{S_x}:= \rho^{-1}(TS_x) \subset T^\ast M|_{S_x}.\]

Since $\rho(T^\ast M) = T\mathcal{F}$ and $TS_x$ is transverse to $T\mathcal{F}$, we have \[ \rho^{-1}(TS_x) =ker(\rho)|_{S_x}.\]

For the cotangent Lie algebroid, this kernel is the conormal bundle of the symplectic foliation: 

\[ ker(\pi^\sharp)|_{S_x} = (T\mathcal{F})^0|_{S_x}.\]
Therefore, \[A_{S_x} = (T\mathcal{F})^0|_{S_x}.\]
Moreover, $A_{S_x}$ has zero anchor. Since the isotropy Lie algebras of the cotangent Lie algebroid of a regular Poisson manifold are abelian, the slice algebroid $A_{S_x}$ is a bundle of abelian Lie algebras over $S_x.$ 

Hence \[ \FHT(A) = \bigsqcup_{x, y \in M} \bigsqcup_{S_x, S_y} Alg((T\mathcal{F})^0|_{S_x}, (T\mathcal{F})^0|_{S_y})).\]

Since the slice algebroids have zero anchor and abelian bracket, an element of \[ Alg((T\mathcal{F})^0|_{S_x}, (T\mathcal{F})^0|_{S_y})\] 

is equivalently a germ of a vector bundle isomorphism \[ (T\mathcal{F})^0|_{S_x} \rightarrow (T\mathcal{F})^0|_{S_y}\]

covering a germ of a diffeomorphism \[ S_x \rightarrow S_y.\]

Now, let $\alpha(t) \in \Gamma(T^\ast M)$ be a compactly supported time-dependent section, and set \[ X_t:= \rho(\alpha(t))= \pi^\sharp(\alpha(t)).\]

Then, $X(t)$ is tangent to the symplectic foliation $\mathcal{F}$.  Let \[ \phi_t\]
denote the flow of $X_t.$ Since $X_t$ is tangent to $\mathcal{F}$, the flow $\phi_t$ preserves the foliation.Equivalently, \[ d\phi_t(T\mathcal{F}) \subset T\mathcal{F.}\]
Dualizing $\phi_t$ induces a map on the conormal bundle: \[ (d\phi_t^{-1})^\ast: (T\mathcal{F})^0|_{S_x} \rightarrow (T\mathcal{F})^0|_{\phi_t(S_x)}.\]
The adjoint flow $\Phi^t_\alpha$ of the cotangent Lie algebroid restricts on the slice algebroid to this induced conormal map:\[ \Phi^t_\alpha|_{A_{S_x}}=(d\phi_t^{-1})^\ast: (T\mathcal{F})^0|_{S_x} \rightarrow (T\mathcal{F})^0|_{\phi_t(S_x)}.\] 

Thus the algebroid holonomy of $\Gamma(A)$-path $(\alpha, p)$ records exactly the transverse holonomy of the leafwise vector field  $X_t = \pi^\sharp(\alpha(t)), $ together with its induced action  on the conormal bundle.

Thus \[ \THT(A)= \bigsqcup_{x \in M} \bigsqcup_{S_x} \{ \Phi^1_\alpha|_{A_{S_x}}:\alpha(t)  \in I_x \Gamma(T^\ast M)\}, \]
Equivalently, using the conormal description of the adjoint flow 
\[ \THT(A)= \bigsqcup_{x \in M} \bigsqcup_{S_x} \{ (d\phi_1^{-1})^\ast|_{(T\mathcal{F})^0|_{S_x}}:\alpha(t)  \in I_x \Gamma(T^\ast M)\}, \phi_t \text{ is the flow of } \pi^\sharp(\alpha(t))\}. \]
Indeed, if $\alpha(t)  \in I_x(\Gamma(T^\ast M)),$ then $ X_t(x) = \pi^\sharp(\alpha(t)_x) = 0,$
so the flow $\phi_t$ fixes $x.$ Hence, $\phi_1(S_x) $ is again a slice through $x$, and the corresponding induced map \[ (d \phi_1^{-1})^\ast: (T\mathcal{F})^0|_{S_x} \rightarrow (T\mathcal{F})^0|_{\phi_1(S_x)}   \]
is precisely a trivial holonomy transformation in the sense of the definition of $\THT(A).$

However, in the regular Poisson case this conormal action is already determined by the ordinary transverse holonomy of the foliation $\mathcal{F}.$ Since the slice algebroids are bundles of abelian isotropy Lie algebras, there is no additional non-abelian isotropy contribution to the holonomy. In particular, the reduced holonomy transformation groupoid  of $T^\ast_\pi M$  agrees with the usual holonomy transformation groupoid of the foliation $\mathcal{F}$:

\[\HT(T^\ast_\pi M) \cong \HT(\mathcal{F}).\]

Consequently, two $\Gamma(T^\ast M)$-paths  $(\alpha, p)$ and $(\beta, p)$ have the same algebroid holonomy if and only if the time-dependent leafwise vector fields \[\pi^\sharp(\alpha(t)) \quad \text{ and } \quad \pi^\sharp(\beta(t))\] 
define the same holonomy transformation of the regular foliation $\mathcal{F}$. Therefore, \[ \Hol(T^\ast_\pi M) \cong \Hol(\mathcal{F}).\]

\end{enumerate}

\section{Smoothness of the Holonomy Groupoid}

In the previous sections, we constructed the holonomy groupoid $\Hol(A)$ of a Lie algebroid and computed it in a variety of examples. A priori, however, $\Hol(A)$ is only a groupoid equipped with a natural quotient structure, and it is not clear whether it admits a smooth structure.

The goal of this section is to show that the obstruction to smoothness comes entirely from isotropy. More precisely, we prove that the restriction of $\Hol(A)$ to each leaf of the characteristic foliation is a Lie groupoid. The proof is based on describing the kernel of the local holonomy map and identifying the corresponding quotient. We then compute the Lie algebroid of $\Hol(A)|_L$ and show that it is obtained from $A|_L$ by quotienting out the strongly central directions.

\subsection{Strategy of the proof}

The starting point is Proposition~\ref{proposition:open-groupoid-morphism-proposition}, which provides  a local integration $\calG$ together with a commutative diagram

\[
\begin{tikzcd}
     & \calG \arrow[r, "\hol_\calG"] \arrow[d, "\pi"] & \Hol(\calG) \arrow[d, hook] \\
    \GammaPath(A) \arrow[rr, "\hol", bend right]\arrow[r] & \Pi_1(A) \arrow[r] & \Hol(A) 
\end{tikzcd}
\] where $hol_{\G}$ is an open local homomorphism. 

Moreover, the  local groupoid homomorphism
\[ 
\hol^\calG \colon \calG \to \Hol(A) 
\] is an open map as 
 it is the composition of two open maps
\[ \calG \to \Pi_1(A) \to \Hol(A). \]

Thus, the smoothness of $\Hol(A)$ can be studied through the kernel \[ \mathcal{K} := \{ g \in \calG \ : \ \hol^\calG(g) \in 1_M \}. \]

If $\mathcal{K}$ were a closed normal local Lie subgroupoid, then one would expect $\Hol(A)$ to be locally obtained as a quotient of $\G$ by $\mathcal{K}$. The main technical ingredient needed to make this precise is the following quotient theorem for local Lie groupoids.

\begin{proposition}\label{prop:quotient-Lie-algebroid-local-integration}
    Let $\widetilde{\calG}$ be a local integration of a Lie algebroid $A$, and let  $\mathcal{K} \grpd M$ be a closed, wide, normal local Lie subgroupoid of $\calG$. Then, after passing to a sufficiently small open restriction  \[\calG \subset \widetilde{\calG},\]  the quotient $\calG/\calK$ is naturally a  local Lie groupoid. Furthermore, its Lie algebroid is canonically isomorphic to the quotient Lie algebroid 
\[A/\Lie(\mathcal{K}).\]
\end{proposition}
\begin{proof}
Let $B = \Lie(\mathcal{K})$ be the Lie subalgebroid of $A$ corresponding to $\calK$.  We define $\mathcal{G}/\mathcal{K}$ as  the quotient of $\G$ by the equivalence relation generated by: 
\[  g_1 \sim g_2 \quad \Longleftrightarrow  \quad \exists  \hspace{0.2cm}k \in \calK\ \text{ s.t. } g_1 \cdot k = g_2.\]
Since elements of $\calK$ need not be invertible, the above relation is not necessarily reflexive. For this reason, we consider the equivalence relation generated by it.
 We do not require that $\calK$ be contained in $\calG$ in this definition since the definition of the quotient is applied to any open restriction $ \calG \subset \widetilde{\calG}$.

Without loss of generality, we may assume that  $M$ is connected and that $\widetilde{\calG} \subset A$ and $\calK \subset B \subset A$ are convenient local groupoids. Consider the  restricted partial multiplication:
$$ \widetilde{\calG} \times_{s,t} \calK \dashrightarrow \widetilde{\calG}  \qquad (g,k) \mapsto g \cdot k.$$

Since $\calK$ is a local Lie subgroupoid, this map is a submersion in a neighborhood of the unit section. Moreover, 
 $1_M \times_{s,t} \calK$ is mapped diffeomorphically onto $\calK$. It follows that  there exists an embedded submanifold  $V \subset \widetilde{\calG}$ containing  the units such that the restricted multiplication map
$$ p \colon V \times_{s,t} \calK \to \widetilde{\calG}$$ is
a diffeomorphism onto its image.

Choose an embedded submanifold $V \subset \widetilde{\calG}$ and an open restriction  $\calG \subset \widetilde{\calG}$ with the following properties:

\begin{enumerate}
    \item whenever $k \in \calK$ and $g \in \calG$ are such that  $g \cdot k $ is well-defined and lies in  $\calG$, the inverse $k^{-1} $ is well-defined in $\mathcal{K}$;
    \item the map $p \colon V \times_{s,t} \calK \to \widetilde{\calG}$ is a diffeomorphism onto its image;
    \item $\calG = \mathrm{Im}(p)$.
\end{enumerate}

The first condition implies that the relation $\sim$ is an equivalence relation on $\G$. The second and the third conditions imply that every element of $\G$ can be written uniquely as \[ g = v \cdot k, \quad \quad  \text{ for    } v \in V, \quad k\in K. \]
The third condition says that $V \subset \mathcal{G}$.

We claim that the natural map
$$ V \to \calG/\calK $$
is a bijection. Injectivity follows from the uniqueness of the above decomposition.  Indeed, if $v_1, v_2 \in V$  represent the same equivalence class, then $v_1 \cdot k = v_2$ for some $k \in \mathcal{K}$. Since $p$ is injective, we must get $k=1$ and therefore $v_1 = v_2$. 
Surjectivity follows from the fact that every element $g \in \G$ admits a decomposition $g = v\cdot k,$ so that $g \sim v.$

Since $ V \cap \widetilde{\calG}$ is an embedded submanifold of $\calG$, we may therefore identify $\G/\calK$ with a smooth manifold $V$. Since $\calK$ is normal, all structure maps descend to the quotient. The necessary groupoid identities  follow from the assumption that $\widetilde{\calG}$ is convenient. Hence $\G/\calK$ inherits a local Lie groupoid structure.

Finally, applying the Lie functor to the sequence 
\[ \calK \cap \calG \hookrightarrow \calG \to \calG/\calK\]
gives 
\[ B \hookrightarrow A \to \Lie(\calG/\calK).\]
Since the second map $A \to \Lie(\calG/\calK)$ is a vector bundle quotient whose kernel is $B$, it follows that \[ \Lie(\G/\calK) \cong A/B = A/\Lie(\calK).\]

 \end{proof}

\begin{remark}(Local quotient by central isotropy)\label{remark:local-quotient-by-central-isotropy} Proposition \ref{prop:quotient-Lie-algebroid-local-integration} shows that, after passing to a sufficiently small open restriction, quotients of local integrations  affects only the isotropy directions. Indeed, the characteristic foliation of $A$ is unchanged by the passage from $A$ to $A/\Lie(\mathcal{K})$, while the isotropy is reduced by quotienting out the Lie subalgebroid $\Lie(\calK)$. Thus the quotient construction preserves the leafwise geometry and removes only infinitesimal isotropy data.

This is precisely the mechanism that will appear in the case of the holonomy groupoid: the kernel of the local holonomy map will correspond to isotropy directions which are invisible to holonomy. Furthermore, once a local model for $\Hol(A)$ is realized as a quotient, its Lie algebroid is determined immediately from the kernel of the quotient map.

\end{remark}

\subsection{The adjoint kernel}

We now turn to the kernel \[ \calK =ker(hol_{\mathcal{G}}).\]

By remark \ref{remark:local-quotient-by-central-isotropy}, understanding this subgroupoid is the key step in applying Proposition \ref{prop:quotient-Lie-algebroid-local-integration}. The main result of this subsection is that $\calK$ admits a purely isotropic description: after passing to a sufficiently small local integration, $\calK$ coincides with the adjoint kernel. In particular, the directions removed by the holonomy quotient are precisely those isotropy directions which act trivially on holonomy.

To formulate this more precisely, we first introduce some notation for the center of a local groupoid.

Given a point $x \in M$, write
\[ Z(\calG)_x := \{ g \in \calG_x \ :  \forall h \in \calG_x \ : gh = hg  \}\]
to denote the center of $\calG$,  and 
$$Z(\calG) := \bigcup_{x \in M} Z(\calG)_x  \subset \calG $$
to denote the union of the centers.

\begin{definition}
    Suppose $\calG$ is a local Lie groupoid. An element $g \in \calG$ is said to  lie in the \emph{adjoint kernel} if there exists a bisection $\sigma$ through $g$ such that $\Ad_\sigma = \Id$ on an open neighborhood of $s(g)$.

    The set of all such element will be called the \emph{adjoint kernel} of $\G.$
\end{definition}

The main goal of this subsection is to relate the adjoint kernel to the kernel of the local holonomy map. We begin by describing the adjoint kernel in terms of the center of a local groupoid. 

\begin{lemma} \label{Lemma:adjoint-identity-center-bisection}
    Suppose $\widetilde{\calG}$ is a local integration of $A$. Then there exists an open neighborhood of the units $\calG \subset \widetilde{\calG}$ with the following property:
    For all bisections $\sigma$ and open sets $U \subset M$, we have that $\Ad_\sigma|_{U} = \Id$ if and only if $\sigma|_U$ takes values only in the center of $\calG$.
\end{lemma}

\begin{proof}
    Choose an open neighborhood $\calG$ of $\widetilde{\calG}$,  with the property that,  for all $x \in M$,  the isotropy group $\calG_x$ admits an open local group homomorphism to the simply connected integration of the isotropy Lie algebra $\mathfrak{g}_x$:
    $$\calG_x \into \calG(\mathfrak{g}_x). $$
    For such a neighborhood,  for all $x \in M$,  the kernel of the adjoint representation of the isotropy group coincides with the center:
    \[ g \in \calG_x, \qquad g \in Z(\calG_x) \Longleftrightarrow \Ad_{g} = \Id_{\mathfrak{g}_x}.\]

    Suppose first that 
    $\sigma$ is a bisection of $\calG$ and  $ \Ad_\sigma|_{U} = \Id \colon A|_U \to A|_U. $
Therefore, $\sigma(x) \in Z(\calG)$ for all $ x \in U$.

Conversely, suppose that $\sigma$ is a local bisection whose image over some open set, say $U$, is contained in $Z(\G).$ Then, for each $x \in U$, the element $\sigma(x) \in Z(\G_x)$ acts trivially by conjugation on the isotropy group $\G_x$. Hence, \[ Ad_{\sigma(x)} = Id_{\mathfrak{g}_x}.\]
On the other hand, the base map of $Ad_{\sigma}$ is the diffeomorphism $t \circ \sigma$. Since $\sigma(x)$ belongs to the isotropy group at $x$, we have \[ (t \circ \sigma)(x) = x,\] so the base map is the identity on $U.$ Therefore, $Ad_{\sigma}$ is  identity on the isotropy Lie algebras. It follows that \[ Ad_{\sigma}|_U = Id.\] 
    
\end{proof}

We now combine to identify the kernel of the local holonomy map with the adjoint kernel. 
 The next lemma records a basic geometric property of central isotropy elements that will be used in the identification of the adjoint kernel.
 
\begin{lemma}\label{lemma:conjugacy-class-embedded}
    Suppose $\widetilde{\calG}$ is a local groupoid. Then there exists an open neighborhood of the units $\calG \subset \widetilde{\calG}$ with the following property: For all $g \in Z(\calG)_x$ the conjugacy class
    $$ \mathfrak{C}(g) := \{ C_h(g)  \ : \ h \in \calG , \ C_{h}(g) \text{ well-defined} \}$$
    is an embedded submanifold and for all $y \in M$ the set $\calG_y \cap \mathfrak{C}(g)$ contains at most one element.
\end{lemma}

\begin{proof}

First we show that, after passing to a sufficiently small open restriction, each isotropy fiber intersects $\mathfrak{C}(g)$ in at most one point.

Let $h$ and $h^\prime$ be arrows with the same target. By the local algebra principle, we may assure that $\G$ is sufficiently small so that \[ C_h = C_{h^\prime} \circ C_{(h^\prime)^{-1}h}\] whenever the expressions are defined. Since $g \in Z(\mathcal{G}_x)$, conjugation by $(h^\prime)^{-1}h$ fixes $g$. Therefore, \[ C_h(g) = C_{h'} \circ C_{(h')^{-1} h} (g) = C_{h'}(g).\]

It follows that the value of $C_h(g)$ depends only on the target of $h.$ Consequently, for every $y \in M$, the set \[ \G_y \cap \mathfrak{C}(g)\]
contains at most one element.

We now show that $\mathfrak{C}(g)$
is an embedded submanifold. Let $\alpha_1,...,\alpha_k \in \Gamma(A)$ be local sections whose values at $x$ form a (local) complement to the isotropy Lie algebra $\mathfrak{g}_x \subset A_x.$ For each $i,$ let $\sigma_i(t)$ denote the local bisection obtained by integrating $\alpha_i.$
Consider the map \[ (t_1,...,t_k) \mapsto C_{\sigma_1(t_1)} \circ \cdot \cdot \cdot \circ C_{\sigma_k(t_k)}(g).\]
Its image lies in $C(g).$ Moreover, the tangent vectors obtained by differentiating with respect to the parameters $t_i$ span the directions transverse to the isotropy. Hence, after restricting to a sufficiently small neighborhood of the origin, the above map defines a local chart on $C(g)$ around $g.$
Therefore, $C(g)$ is an embedded submanifold.
\end{proof}

We now combine Lemmas  \ref{Lemma:adjoint-identity-center-bisection} and \ref{lemma:conjugacy-class-embedded}   to identify the kernel of the local holonomy map with the adjoint kernel.

\begin{proposition}\label{proposition:local-integration-adjoint-kernel}
Suppose $A$ is a Lie algebroid. There exists a local integration $\calG$ with holonomy such that $\mathcal{K} := \ker (\hol^\calG)$ is equal to the adjoint kernel.
\end{proposition}

\begin{proof}
One direction is clear, elements of the adjoint kernel must have trivial holonomy from the definition of $\hol^\calG$.

For the other direction, choose a local integration $\G$ of $A$ satisfying Lemmas \ref{Lemma:adjoint-identity-center-bisection} and \ref{lemma:conjugacy-class-embedded}. We may further assume $\G \subset A$ is source-linear and convex.

Let $g \in \G$ be such that $hol_\mathcal{G}(g) = 1_x.$ Since $\calG$ is convex,  there exists a time-dependent section $\alpha \in \Gamma^t_c(A)$ with $\calG$-integration $\sigma$ such that $\sigma(1)(x) = g$. Let $S_x$ be a slice through $x = s(g)$. By Corollary \ref{corollary:open-neighborhood-identity-holonomy},   \[ hol(\alpha,x) = hol_{\G}(\sigma(1)(x)) = hol_{\G}(g) = 1_x.\] Therefore, the holonomy transformation represented by $\Phi^1_\alpha|_{A_{S_x}}$ is trivial.   Thus,  there exists \[ \epsilon(t) \in I_x \Gamma_c^t(A)\] such that \[ \Phi^1_\alpha|_{A_{S_x}} = \Phi^1_\epsilon|_{A_{S_x}}.\]

Let $\eta$ denote the $\G$-integration of $\epsilon$, and define

\[  \widetilde{\sigma} := \eta(1)^{-1} \sigma(1). \]
After shrinking $\calG$ if necessary, by  local algebra principle, we get
\[ \Ad_{\widetilde{\sigma}}|_{S_x} = \Id. \]

Lemma \ref{Lemma:adjoint-identity-center-bisection} now implies that 
 $\widetilde{\sigma}$ takes values in $Z(\calG)$ along $S_x$. By Lemma \ref{conjugacy-smooth-local-section-family}, applied to the smooth section $\tilde{\sigma}$ after shrinking around $x$ if necessary, there exists a smooth local section $\bar{\sigma}:U \rightarrow Z(\G)$ defined on an open neighborhood $U \subset M$ of $x$, such that \[ \bar{\sigma}|_{U \cap S_x} = \tilde{\sigma}|_{U \cap S_x}.\] Since $\bar{\sigma}$ takes values in $Z(\G)$, Lemma \ref{Lemma:adjoint-identity-center-bisection} implies  that $\Ad_{\bar{\sigma}}= Id$ on $U$.  Thus, every element in the image of $\bar{\sigma}$ belongs to the adjoint kernel. In particular, \[ g  = \sigma(1) (x) = \tilde{\sigma}(x)  = \bar{\sigma}(x)\] lies in the adjoint kernel.  
 
 We conclude that \[ ker(hol_\calG) = \text{ adjoint kernel }.\]
 
\end{proof}

\begin{example}(Lie groups) Let $A= \mathfrak{g} \rightarrow \{\ast\}$ be a Lie algebra and let $G$ be a local Lie group integrating $\mathfrak{g}$. Since the base consists of a single point, the holonomy groupoid is the image of the adjoint representation.  In this case, the adjoint kernel in Proposition \ref{proposition:local-integration-adjoint-kernel} is the local kernel of the adjoint representation. Thus Proposition \ref{proposition:local-integration-adjoint-kernel} reduces to the familiar fact that elements with trivial adjoint action are precisely the elements which act trivially on infinitesimal holonomy.
    
\end{example}

\begin{theorem}\label{theorem:hol(A)-is-long-smooth}
    Suppose $A$ is a Lie algebroid and let $\calO$ be an orbit of $A$. There exists a local integration $\calG$ for which $\ker (\hol^\calG)|_{\mathcal{O}} $ is a closed, normal Lie subgroupoid. Therefore, $\Hol(A)|_{\calO}$ is a Lie groupoid.
\end{theorem}

\begin{proof}
 Let  $\mathcal{K} := \ker (\hol^\calG)$. By  proposition \ref{proposition:local-integration-adjoint-kernel}, after passing to a sufficiently small local integration $\G$, an element $g \in \mathcal{G}$  lies in $\calK$ if and only if $g$ can be extended to a local bisection $\sigma$ taking values in $Z(\calG)$ in a neighborhood of $x$. In particular, $\mathcal{K} \subset Z(\calG)$. We may also assume that that $\calG \subset A$ is source-linear and convex local integration. 

Let  $Z(A) \subset A$  denote the union of the centers of the isotropy Lie algebras of $A$,  and let 
    $K \subset Z(A)$ be the subset consisting of those elements which admit a local extension to  sections of $Z(A)$. Since $Z(A)$ is a linear subset of $A$ and  is preserved under adjoint flows, it follows that $K|_{\calO}$ is a smooth subbundle of $A|_\calO$.

  Next, choose a linear integration $\calG$ so that $Z(\calG)$ is connected. This is possible since $Z(\calG)$ is a closed subset of $\calG$. Since $\G$ is source-linear, the exponential map on the isotropy directions, $Z(A) \dashrightarrow Z(\calG)$,  is the identity. Hence
    \[  Z(\calG) = (Z(A) \cap \calG) \subset A\]

   Therefore, an element of $Z(\calG)$ can be extended to a section taking values in $Z(\calG)$ if and only it can be extended to take values in $Z(A)$. Consequently,
    \[ \calK = (K \cap \calG) \subset A \]
    Since $K|_{\calO} \subset A|_\mathcal{O}$ is a vector bundle, is is closed. Because $\G$ is an open neighborhood of the zero section,  it follows that $\calK|_{\calO} \subset \calG|_{\calO}$ is closed. By construction $\calK|_\calO$ is wide and normal.

    Proposition \ref{prop:quotient-Lie-algebroid-local-integration} therefore applies and shows that the quotient \[ \G|_\calO/\calK_\calO \]
is a local Lie groupoid.  Since $hol_\calG$ is an  open local groupoid homomorphism with kernel $\calK$, this quotient agrees locally with $\Hol(A)|_\calO.$ Hence $\Hol(A)|_\calO$ is a Lie groupoid. \end{proof}
\begin{example}
    For $A=TM$, Theorem \ref{theorem:hol(A)-is-long-smooth} recovers the fact that the pair groupoid on each connected component is a Lie groupoid. For $A=\F \subset TM$ a regular foliation, it recovers the longitudinal smoothness of the usual holonomy groupoid. For $A=\mathfrak{g} \rightarrow \{\ast\},$ it reduces to the smoothness of $Inn(\mathfrak{g}).$
\end{example}

\subsection{The holonomy Algebroid}

In this section, we will compute the  algebroid of $\Hol(A)$. By Proposition \ref{prop:quotient-Lie-algebroid-local-integration}, it suffices to identify the infinitesimal counterpart of the adjoint kernel. This leads to the definition of holonomy algebroid. 

\begin{definition}\label{definition:strongly-central}
    Suppose $A$ is a Lie algebroid. We say an element $z \in A$ is \emph{strongly central} if there exists a section $\alpha \in \Gamma(A)$, extending $z$, such that $\alpha$ takes values in $Z(A)$.\\
    We denote by  $\SC(A)_x$ the subset of strongly central elements in $A_x$,  and write $\SC(A)$ to denote the set of all strongly central elements.
    
\end{definition}

For each $x \in M$, the set 
$\SC_x(A)$ is a linear subspace of $A_x$. Moreover, $\SC(A)$ is preserved by adjoint flows.  Proposition \ref{proposition:local-integration-adjoint-kernel} suggests that $\SC(A)$ should be viewed as the infinitesimal counterpart of the adjoint kernel.

\begin{remark}
    In general, $\SC(A)_x$ should not be confused with the center $Z(\mathfrak{g}_x)$ of the isotropy Lie algebra. An element may be central in the single fiber $\mathfrak{g}_x$ without extending locally to a section which commutes with all sections of $A.$ Thus in general \[ \SC(A)_x \subset Z(\mathfrak{g}_x)\] and the inclusion may be strict.
\end{remark}

\begin{definition} \label{definition:holonomy-algebroid}
    Suppose $A$ is a Lie algebroid. The \emph{holonomy algebroid} of $A$ is the quotient 
    $$\HAlg(A) := A/\SC(A).$$
 
\end{definition}

Although $\HAlg(A)$ need not be smooth vector bundle globally, as the rank of $\SC(A)$ can vary. It is the natural infinitesimal candidate associated to holonomy. The next theorem shows that, along each leaf, it is precisely the Lie algebroid of the holonomy groupoid.

\begin{theorem}\label{theorem:holonomy-algebroid-isomorphism}
The Lie algebroid of \( \Hol(A)|_L \) is is canonically isomorphic to the holonomy algebroid \( \HAlg(A)_L \).
\end{theorem}

\begin{proof}
    By Theorem \ref{theorem:hol(A)-is-long-smooth}, $\Hol(A)|_L$ is a Lie groupoid. Choose a local integration $\G$ of $A$ and let \[hol_{\G}:\G \rightarrow \Hol(A)\]

    be the local holonomy homomorphism. Let \[K:= ker(hol_{\G}).\]
    By Theorem \ref{theorem:hol(A)-is-long-smooth}, after replacing $\G$ by a sufficiently small open restriction around the unit section if necessary,  $K|_L$ is a closed normal Lie subgroupoid of $\G|_L$. Moreover, $K|_L$ is wide, since $K$ is the kernel of the groupoid morphism $hol_{\G},$ and hence contain every unit $1_x$ for $x \in L.$ Therefore, $K|_L$ is a closed, wide, normal local Lie subgroupoid of $\mathcal{G}|_L.$

Hence Proposition \ref{prop:quotient-Lie-algebroid-local-integration} applies and gives a local Lie groupoid quotient \[ \mathcal{G}|_L/K|_L.\]
    The holonomy map \[ hol_{\G}:\G|_L \rightarrow \Hol(A)|_L\]
    has kernel $K|_L,$ and therefore descends to a local groupoid morphism \[ \overline{hol}_{\G}:\G|_L/K|_L \rightarrow \Hol(A)|_L.\]
    By the construction of the smooth structure on $\Hol(A)|_L$ in the proof of Theorem \ref{theorem:hol(A)-is-long-smooth}, this descended map identifies a neighborhood of the unit section in $\G|_L/K|_L$ with a neighborhood of the unit section in $\Hol(A)|_L.$ Therefore, for the purpose of computing the Lie algebroid,
    \[ \Lie(\Hol(A)|_L) \cong \Lie(\G|_L/K|_L).\]

By Proposition \ref{prop:quotient-Lie-algebroid-local-integration}, the Lie algebroid of this quotient is \begin{equation}\label{equation:Quotient-Algbroid-Iso}
\Lie(\Hol(A)|_L) \cong \Lie(\G|_L/K|_L) \cong A|_L / \Lie(K|_L).    
\end{equation}
By Proposition \ref{proposition:local-integration-adjoint-kernel}, for the chosen local integration \[K=ker(hol_{\G}),\]
is the adjoint kernel.

 By the characterization of strongly central elements, the strongly central elements are precisely the infinitesimal kernel of the holonomy map.  Hence, 
 \[ \Lie(K|_L) =\SC(A)|_L.\]

 Substituting this in the equation \ref{equation:Quotient-Algbroid-Iso}   gives 
\[ \Lie(\Hol(A)|_L) \cong A|_L /\SC(A)|_L.\]

This is precisely the holonomy algebroid of $A$ along $L.$

\end{proof}

\begin{example}(Lie algebras)
    Let $A= \mathfrak{g} \rightarrow \{\ast \}$. By Example \ref{example:Lie-algebra}, the algebroid holonomy groupoid is \[\Hol(\mathfrak{g}) = Inn(\mathfrak{g}),\]
    the image of the adjoint representation. Hence its Lie algebra is \[ \Lie(\Hol(\mathfrak{g})) = \ad(\mathfrak{g}) \cong \mathfrak{g}/Z(\mathfrak{g}).\]

    On the other hand, since the base is a point, a section of $A$ is just a point of $A.$ The set of strongly central elements are precisely the elements of the center  \[ \SC(A) = Z(\mathfrak{g}).\]
    Thus Definition \ref{definition:holonomy-algebroid} gives \[ \HAlg(A) = \mathfrak{g}/Z(\mathfrak{g}).\]
    Therefore Theorem \ref{theorem:holonomy-algebroid-isomorphism} gives \[ \Lie(\Hol(\mathfrak{g})) \cong \HAlg(\mathfrak{g}).\]

\end{example}

\begin{example}
    (Trivial bundles of Lie algebras) Let \[ A = M \times \mathfrak{g} \rightarrow M\]
be the Lie algebroid with zero anchor and constant isotropy Lie algebra $\mathfrak{g}$. By example \ref{free-transitive-lie algebroid-holonomy}, the holonomy groupoid is governed by the inner automorphism group of $\mathfrak{g}$. Hence, over each point $x \in M,$ the Lie algebra of the holonomy group is \[ \mathfrak{g}/Z(\mathfrak{g}).\]
On the other hand, the strongly central elements are \[ \SC(A) = M \times Z(\mathfrak{g}).\]
    Therefore Definition \ref{definition:holonomy-algebroid} gives \[ \HAlg(A) = (M \times \mathfrak{g}) /(M \times Z(\mathfrak{g})) \cong M \times (\mathfrak{g}/Z(\mathfrak{g}).\]
    This agrees with Theorem \ref{theorem:holonomy-algebroid-isomorphism}.
\end{example}

\begin{example}(Tangent algebroids) Let \[ A= TM \rightarrow M\] with anchor $\rho = id_{TM}$. By Example \ref{example:holonomy-algebroid-tangent-bundle}, \[ \Hol(TM) = \{(x,y) \in M \times M: x, y \text{ lie in the same connected component}\}.\] 
    Thus, on each leaf $L,$ we have \[ \Hol(TM)|_L = L \times L,\]
    and hence \[ \Lie(\Hol(TM)|_L) = TL.\]

    Since the leaf is a connected component of $M,$ this is simply 
\[TL= TM|_L.\]
Moreover, $TM$ has zero isotropy, so \[\SC(TM) = 0.\]
Therefore, Definition \ref{definition:holonomy-algebroid} gives \[ \HAlg(TM)_L = TM|_L.\]
Hence Theorem \ref{theorem:holonomy-algebroid-isomorphism} gives \[ \Lie(\Hol(TM)|_L) \cong \HAlg(TM)_L.\]

\end{example}

\begin{example}(Transitive Lie algebroids) Let $A \rightarrow M$ be a transitive Lie algebroid. Then $M$ itself is the only leaf. Write\[ \mathfrak{g}_M:= ker(\rho)\] for the isotropy Lie algebra bundle. By example \ref{example:algebroid-holonomy-transitive}, the slices are points, so the holonomy transformations are given by the induced isomorphisms between the isotropy Lie algebras \[ \mathfrak{g}_x \rightarrow \mathfrak{g}_y.\]
Hence, the infinitesimal kernel of the holonomy consists of those isotropy directions where local adjoint action is trivial. By definition \ref{definition:strongly-central}, these are precisely the strongly central elements $\SC(A) \subset \mathfrak{g}_M.$ 

Hence, using Definition \ref{definition:holonomy-algebroid}, \[\HAlg(A) = A/\SC(A), \quad \SC(A) \subset \mathfrak{g}_M.\]
    Therefore Theorem \ref{theorem:holonomy-algebroid-isomorphism} identifies \[ \Lie(\Hol(A)) \cong A/\SC(A).\]
    Thus the holonomy algebroid is obtained from $A$ by quotienting out the ineffective strongly central isotropy directions.
\end{example}
\begin{example}(Regular foliations)
 Let $\F \subset TM$ be a regular foliation, viewed as a Lie algebroid with anchor the inclusion \[ \F \xhookrightarrow{} TM.\]
 By Example \ref{example:algebroid-holonomy-foliations},  the algebroid holonomy groupoid $\Hol(\F)$ is the usual holonomy groupoid of the foliation. If $L$ is a leaf, then \[ \Lie(\Hol(\F)_L) = \F|_L.\]

Since the anchor $\F \hookrightarrow TM$ is injective, the isotropy is zero, and therefore \[ \SC(\F) = 0.\]
Thus Definition \ref{definition:holonomy-algebroid} gives \[ \HAlg(\F)_L= \F|_L.\]
Hence Theorem \ref{theorem:holonomy-algebroid-isomorphism} gives \[ \Lie(\Hol(\F)|_L) \cong \F|_L = \HAlg(\F)_L.\]
 
\end{example}

\begin{example}
    (Regular Poisson manifolds) Let $(M, \pi)$ be a regular Poisson manifold, and let \[ A= T^\ast_\pi M\] be its cotangent Lie algebroid. The anchor is \[ \pi^\sharp: T^\ast M \rightarrow TM,\] and its image is the symplectic foliation $\F.$ Thus, for a symplectic leaf $L$, \[ A|_L = T^\ast M|_L.\]
By Example \ref{example:algebroid-holonomy-regular-poisson}, the holonomy of $T^\ast_\pi M$ agrees with the holonomy of the regular foliation $\F.$ Hence the Lie algebroid of \[ \Hol(T^\ast_\pi M)|_L\]
is \[ T\F|_L.\]
 The isotropy bundle of $T^\ast_\pi M$ is \[ ker(\pi^\sharp) = (T\F)^0.\]

 We claim that 
   \[\SC(T^\ast_\pi M)|_L=(T\F)^0|_L.\]

Indeed, a strongly central element must lie in the isotropy, so \[ \SC(T^\ast_\pi M)|_L \subset (T\F)^0|_L.\]

Conversely, because $\F$ is regular, every point admits local coordinates adapted to the foliation. In such coordinates the conormal bundle is locally spanned by differentials of transverse coordinates. These conormal sections commute with all sections of $T^\ast_\pi M$, and therefore every conormal element locally extends to a strongly central section. Hence \[ (T\F)^0|_L \subset \SC(T^\ast_\pi M)|_L.\]
Thus \[\SC(T^\ast_\pi M)|_L = (T\F)^0|_L.\]
   
Therefore, by    Definition \ref{definition:holonomy-algebroid},  \[ \HAlg(T^\ast_\pi M)_L = T^\ast M|_L/(T\F)^0|_L \cong T^\ast \F|_L.\]
Using the leafwise symplectic form, $T^\ast \F|_L$ is canonically identified with $T\F|_L$. Therefore Theorem \ref{theorem:holonomy-algebroid-isomorphism} is consistent with Example \ref{example:algebroid-holonomy-regular-poisson}: \[\Lie(\Hol(T^\ast_\pi M)|_L) \cong \HAlg(T^\ast_\pi M)_L.\]

\end{example}

\subsection{Holonomy as an adjoint integration}\label{section:universal-property-holonomy-groupoid}

In Theorem~\ref{theorem:holonomy-algebroid-isomorphism} we saw that the holonomy algebroid is the Lie algebroid of the holonomy groupoid when restricted to a single leaf.
The restriction to the leaf caveat exists because the set of strongly central elements of $A$ does not necessarily form a subbundle.
However, the set of strongly central elements does form vector subspace within each fiber of $A$ and the only obstruction to smoothness is the rank.

In other words, an immediate corollary of Theorem~\ref{theorem:holonomy-algebroid-isomorphism} is:
\begin{corollary}\label{corollary:holonomy-constant-rank}
Suppose $A$ is a Lie algebroid and the set of strongly central elements of $A$ has constant rank. Then $\HAlg(A)$ is a Lie algebroid, $\Hol(A) \grpd M$ is a Lie groupoid, and $\Lie(\Hol(A)) \cong \HAlg(A)$.
\end{corollary}
If the set of strongly central elements forms a trivial subbundle, then $\HAlg(A) \cong A$ and so $\Hol(A)$ is a Lie groupoid integrating $A$ itself.
Indeed, as an integration of $A$, it is special in that it satisfies a universal property which is ``opposite'' to the one satisfied by the source simply connected integration.

\begin{theorem}\label{theorem:holonomy-adjoint-integration}
Suppose the set of strongly central elements of $A$ is trivial. 
If $\calG$ is a source-connected integration of $A$, then there exists a unique groupoid homomorphism $F \colon \calG \to \Hol(A)$ such that $\Lie(F) = \Id_A$.
\end{theorem}
\begin{proof}
In Corollary~\ref{Cor-Groupoid-Hol-Equal-Algeb-Hol}, we saw that the groupoid holonomy of a source-connected integration of $A$ coincides with the algebroid holonomy. 
This, combined with Theorem~\ref{theorem:holonomy-algebroid-isomorphism} establishes existence.

For uniqueness, suppose $F \colon \calG \to \Hol(A)$ is a groupoid homomorphism such that $\Lie(F) = \Id_A$. 
Let us fix $g \in \calG$. 

Recall that the holonomy of $g$ is computed by choosing a local bisection $\sigma \colon U \to \calG$ through $g$ and then considering the class of $\Ad_\sigma \colon A|_U \to A$.
Since $F$ is a groupoid homomorphism covering the identity we also obtain a local bisection $F(\sigma)$ through $F(g)$ in $\Hol(A)$.
Furthermore, a standard calculation shows that:
\[ \Ad_{F(\sigma)} = \Lie(F) \circ \Ad_\sigma \circ \Lie(F)^{-1} \]
But since $\Lie(F) = \Id_A$, we have that $\Ad_{F(\sigma)} = \Ad_\sigma$.

Hence, $F(g)$ and $g$ have the same holonomy. Since $F(g)$ is an element of $\Hol(A)$ itself, its holonomy is itself and so $F(g) = \hol(g)$.
\end{proof}

\appendix
\section{Flows of derivations}\label{appendix:derivations}
In this appendix we will detail some of the basic theory of the flow of a derivation. 
We will see that derivations of vector bundles are the infinitesimal version of the flow of a one parameter family of vector bundle automorphisms.
\subsection{Derivations}
\begin{definition}\label{definition:derivation}
A \emph{derivation} on a vector bundle \( E \to M \) consists of a pair \( (D,X) \) where \( D \colon \Gamma(E) \to \Gamma(E) \) is a linear map and \( X \in \mathfrak{X}(M) \) is a vector field such that the following Leibniz identity holds:
    \begin{equation}\label{eqn:derivation.leibnitz} 
        D(u\eta) = X(u)\eta + u D(\eta)
    \end{equation}
    for all \( u \in C^\infty(M) \) and \( \eta \in \Gamma(E) \).
    The vector field \( X \) is called the \emph{symbol} of the derivation \( D \) and we may sometimes abuse notation and represent the derivation \( (D,X) \) simply by \( D \).
\end{definition}
Let us look at three key examples of derivations.
\begin{example}[Vector fields]\label{example:derivation.vector.fields}
    A vector field \( X \in \mathfrak{X}(M) \) can be seen as a derivation \((X,X)\) on the trivial line bundle \( \underline{\R}_M \to M \). 
    Sections of this bundle are smooth functions on \( M \) and Equation~\ref{eqn:derivation.leibnitz} holds reduces to the usual Leibniz rule for vector fields.
\end{example}
\begin{example}[Connections]\label{example:derivation.connections}
    A connection \( \nabla \) on a vector bundle \( E \to M \) can be seen as a derivation \( (\nabla_X, X) \) for each vector field \( X \in \mathfrak{X}(M) \). 
    The Leibniz rule for connections is just the usual Leibniz rule for connections.
\end{example}

The next example is the most important for our purposes. It shows that every one-parameter family of vector bundle automorphisms determines a derivation. In the following subsection we shall prove the converse.

\begin{lemma}\label{lemma:derivation.from.automorphisms}
    Suppose \( \Phi^t \colon E \to E \) is a one parameter family of vector bundle automorphisms covering a one parameter family of diffeomorphisms \( \phi^t \colon M \to M \) with \( \Phi^0 = \Id_E \).
    Then the following formula defines a derivation on \( E \):
\[ D \colon \Gamma(E) \to \Gamma(E) \qquad D(\eta) := -\left. \frac{d}{dt} \right\rvert_{t=0} \Phi^t_* \eta.  \]
    The symbol of this derivation is the vector field:
\[ X \in \mathfrak{X}(M) \qquad X_p = \left. \frac{d}{dt} \right\rvert_{t=0} \phi^t(p).  \]
\end{lemma}

\begin{proof}
    We first recall that the push-forward on sections is compatible with multiplication by functions. Namely, for every $u \in C^\infty(M)$ and $\eta \in \Gamma(E)$ one has \[ (\Phi^t)_\ast(u\eta) = (\phi^t)_\ast  u \cdot (\Phi^t)_\ast \eta.\] Let \[ X := \left.\frac{d}{dt}\right|_{t=0}\phi^t\] denote the infinitesimal generator of the base flow. Differentiating the identity \[ (\phi^t)_\ast u(p) = u(\phi^{-t}(p))\]

at $t = 0$ gives \[ \left.\frac{d}{dt}\right|_{t=0}(\phi^t)_\ast u = -X(u).\]

We now compute: \begin{align*}
D(u \eta) & = -\left. \frac{d}{dt} \right\rvert_{t=0} \Phi^t_* (u \eta)  \\
& = -\left. \frac{d}{dt} \right\rvert_{t=0} \left( (\phi^t)_* u \cdot (\Phi^t)_* \eta \right)  \\
& = - u \cdot \left. \frac{d}{dt} \right\rvert_{t=0} (\Phi^t)_* \eta - \left. \frac{d}{dt} \right\rvert_{t=0} (\phi^t)_* u \cdot \eta  \\
& = u D(\eta) + X(u) \eta.
\end{align*}
Hence $(D, X)$ satisfies the Leibniz rule and therefore defines a derivation of $E$ with symbol $X.$
    
\end{proof}

This derivation may be regarded as infinitesimal generators of one-parameter families of vector bundle automorphisms.
\subsection{Flow of a derivation}\label{subsection:flow.of.derivation}
Earlier we saw that the derivative of a one parameter family of vector bundle automorphisms gives rise to a derivation. 
Modulo some assumptions about completeness, it turns out that every derivation arises this way.
\begin{definition}\label{definition:complete.derivation}
    Suppose \( (D,X) \) is a derivation on \( E \to M \). We say that \( (D,X) \) is \emph{complete} if the vector field \( X \) is complete.
\end{definition}
\begin{definition}[Flow of a derivation]\label{definition:flow.of.derivation}
    Suppose \( (D,X) \) is a derivation on \( E \) and \( X \) is a complete vector field. Given any section \( \eta_0 \in \Gamma(E) \) we say that a one parameter family of sections \( \eta(t) \) is the \emph{flow} of \( \eta_0 \) along \( D \) if it satisfies the following initial value problem:
    \begin{equation}\label{eqn:flow.of.derivation.IVP} \eta'(t) = -D\eta(t) \qquad \eta(0) = \eta_0. \end{equation}
\end{definition}
\begin{proposition}\label{prop:derivation.flow.exists.and.unique}
    Suppose \( (D,X) \) is a complete derivation and \( \eta_0  \in \Gamma(E) \) is a section. Then the flow of \( \eta_0 \) along \( D \) exists and is unique.
\end{proposition}

For the proof of the above proposition, we need the following lemma.

\begin{lemma}\label{lemma:derivation.IVP.compatible.with.multiplication}
    Suppose \( (D,X) \) is a derivation, \( X \) is a complete vector field, and \( \eta(t) \) is a one parameter family of sections satisfying:
    \[ 
    \eta'(t) = -D \eta(t).
    \]
    Given a smooth function \(  u_0 \in C^\infty(M) \), define a one parameter family of smooth functions by taking \( u(t):= (\phi^t_X)_*(u_0) \). Then we have that:
    \[ \frac{d}{dt} (u(t) \eta(t)) = -D( u(t) \eta(t)).\]
\end{lemma}
\begin{proof}

Recall from the definition of $u(t)$ we have  \( \frac{d}{dt} u(t) = - X(u(t)) \).


Therefore,
\begin{align*}
\frac{d}{dt} (u(t) \eta(t)) & = u'(t) \eta(t) + u(t) \eta'(t)  \\
& = -X(u(t)) \eta(t) - u(t) D(\eta(t)) \\
& = -D(u(t) \eta(t)).
\end{align*}
\end{proof}
\begin{proof}[Proof of Proposition~\ref{prop:derivation.flow.exists.and.unique}]

Since it suffices to construct flows locally in \( M \), we can assume without loss of generality that \( E = \R^k \times M\) and \( \Gamma(E) = C^\infty(M, \R^k)\),  for some natural number \( k \). 
Let $\nabla$  denote the canonical flat connection on \( E \). Observe that  there must exist a smooth function \( \Omega \in C^\infty(M, \End(\R^k)) \) such that 
\[  D(\eta) = \Omega \eta + \nabla_X \eta.\]
We will begin by showing that the flow of a section along a derivation exists when \( \eta_0 \colon M \to \R^k \) is a constant function. 
Notice that if \( \eta \in C^\infty(M,\R^k ) \) is any constant function then
\[ D(\eta) = \Omega \eta.\]
By the uniqueness and existence theorem for linear ODEs, there must exist a unique one-parameter family of functions \( A(t) \colon M \to \GL(\R^k) \) satisfying
\[ A'(t)  = -\Omega A(t), \qquad \quad  A_0 = \mathbb{1}_k.  \]
Given any constant function \( \eta_0 \in C^\infty(M,\R^k) \),  we can define:
\[ \eta(t) := A(t)\eta_0. \]

Note that \( \eta(t) \) is a one parameter family of sections of \( E \).  Furthermore, for each \( t \in \R \) each section   \( \eta(t) \colon M \to \R^k \) is a constant function. Doing a  
direct calculation we get
\[ \eta'(t) = - D\eta(t), \qquad \eta(0)  = \eta_0.  \]
Hence
 \( \eta(t) \) is the flow of \( \eta \) along \( (D,X) \). 

Since every function in \( C^\infty(M,\R^k \)) is  \( C^\infty(M) \)-linear combination of constant functions and in light of Lemma~\ref{lemma:derivation.IVP.compatible.with.multiplication} it follows that flows exist for arbitrary sections.

We now prove uniqueness. Suppose \( \eta(t) \) satisfies the property:
\[ \eta'(t) = -D \eta(t),   \qquad \eta(0)= \eta_0.\]
Let \( p_0 \in M \) be an arbitrary point and let \( \gamma(t) \) be the unique integral curve of \ \(X \) with \( \gamma(0)= p_0\). Consider the function
\[ v(t) \colon \R \to \R^k \qquad v(t) := \eta(t)(\gamma(t)). \]
Differentiate it to get
\begin{align*} v'(t) &=  \eta'(t)(\gamma(t)) + \nabla_X \eta(t)(\gamma(t)) \\ & = \left( - \Omega \eta(t) - \nabla_X \eta(t) \right)(\gamma(t)) + (\nabla_X \eta(t))(\gamma(t))\\
&= -\Omega(\gamma(t)) v(t). \end{align*}
In other words, \( v(t)\) satisfies the following initial value problem
\begin{equation}\label{eqn:derivation.transport.equation} v'(t) = -\Omega(\gamma(t))v(t), \qquad v(0) = \eta_0(p_0).   \end{equation}
This is a (time-dependent) linear ODE.   By the previously mentioned theorems, such a linear ODE admits a unique solution. 
In other words, \( \eta(t)(\gamma(t)) \) is uniquely determined by \( \eta_0(p_0)\). Since \( p_0 \) is arbitrary, it follows that \( \eta(t) \) is uniquely determined by \( \eta_0  \).
\end{proof}
\begin{corollary}\label{cor:derivation.flow.bundle.automorphisms}
    Given a complete derivation \( (D,X) \) on \( E \to M \), 
    there exists a unique one-parameter family of vector bundle automorphisms, denoted by,
    \[ \Phi^t_D \colon E \to E \]
    covering the flow of \( X \), \( \phi^t_X \colon M \to M \),  satisfying the property that for all \( \eta_0 \in \Gamma(E) \) the push-forwards

\[ \eta(t):= (\Phi^t_D)_*(\eta_0) \]
is the flow of \( \hat \eta \) along \( (D,X)\). We call \( \Phi^t_D \) the (point-wise) flow of the derivation \( (D,X) \). 
\end{corollary}
\begin{proof}
In order to  to utilize the setup of the previous proof we assume 
   \( E = \R^k \times M \) and sections of \( E \) are \( \R^k \)-valued functions. Let\( p \) be an arbitrary point  of \( M \). Consider the following initial value problem for one-parameter functions \( P(t) \colon \R \to \GL(\R^k) \):
\[ P'(t) = -\Omega(\phi^t_X(p)) \cdot P(t), \qquad  \qquad P(0) = 1. \]
The standard theorems for existence and uniqueness of linear ODEs ensures a solution for every point $p \in M$ for the above initial value problem.

We define \( \Phi^t_D \) uniquely by requiring that it satisfy:
\[ \Phi_D^t \colon \R^k \times M \to \R^k \times M, \qquad  \Phi^t_D (e,p) = (P(t) e, p).\] 
From this definition, it is immediate that \( \pi \circ \Phi^t_D = \phi^t_X \circ \pi \), where \ \( \pi \colon E \to M \) is the base projection. Moreover,  \(\Phi^t_D \) is a linear isomorphism on each fiber since each  \( P(t) \) is an  invertible matrix.   We conclude that we have successfully defined a one-parameter family of vector bundle automorphisms.

When evaluating \(  \eta(t): = (\Phi^t_X)_* \eta_0 \) along an integral curve, it satisfies Equation~\ref{eqn:derivation.transport.equation}.  Therefore, for all 
 \( \eta_0 \colon M \to \R^k \), we have that \( (\Phi^t_X)_* \eta_0 \) is the unique flow of \( \eta_0 \) along \( D \). Moreover, we also observed in that proof, the Equation~\ref{eqn:derivation.transport.equation}  uniquely determines the values of the flow of a section.

\end{proof}
\subsection{Transport vector field and time-dependent derivations}
The flow of a derivation, is a one-parameter family of diffeomorphism. Hence it defines a vector field on the total space of the vector bundle. We will call this vector field the transport vector field of the derivation.

Derivations can also be time-dependent. So long as the flow of the underlying vector field exists, the flow of a time-dependent derivation exists as well.
\begin{definition}
    A \emph{time-dependent derivation} on a vector bundle \( E \to M \) consists of a one-parameter family of derivations \( (D(t), X(t)) \),  where \( D(t) \colon \Gamma(E) \to \Gamma(E) \). We say a time-dependent derivation is \emph{complete} if the underlying time-dependent vector field \( X(t) \) is complete.

    Given \( \eta_0 \in \Gamma(E) \),  we define the flow of \( \eta_0 \) along \( (D(t),X(t)) \) to be a one-parameter family of sections \( \eta(t) \) satisfying the following initial value problem:
    \begin{equation}\label{eq:time.dependent.derivation.flow} \eta'(t) = -D(t) \eta(t), \qquad \eta(0) = \eta_0. \end{equation}
\end{definition}
\begin{proposition}
    The flow of a complete time-dependent derivation exists and is unique.
\end{proposition}
\begin{proof}
Any time-dependent derivation \( (D(t),X(t)) \) can be made into a time-independent derivation \( (\overline{D},\overline{X} ) \) on the vector bundle \( E \times \R \to M \times \R \) by taking:
\[ \overline{D}(\eta)_{(p,t)} = D(t)(\eta|_{\{ t \} \times M})_p, \qquad \overline{X}(p,t) = (X_t(p), \partial_t). \]\
If \( X(t) \) is complete then \( \overline{X} \) is complete. Hence, the flow of \( (\overline{D},\overline{X}) \) exists and is unique. 

Given a section \( \eta_0 \in \Gamma(E) \),  we can define a section \( \overline{\eta}_0 \in \Gamma(E \times \R) \) as follows:
\[  \overline{\eta}_0(p,t) = (\eta_0(p),t). \]
This section will have an adjoint flow \( \overline{\eta}(t) \in \Gamma(E \times \R) \).

From this we can define the flow of \( \eta_0 \) by taking
\[ \eta(t) := \pi_E \circ \overline{\eta}(t) \]
where \( \pi_E \colon E \times \R \to E\) is the projection onto the first factor. 

We leave it to the reader to verify that \( \eta(t) \) indeed satisfies \eqref{eq:time.dependent.derivation.flow} and is unique.
\end{proof}

\begin{lemma}
    The flow of a time-dependent section of a Lie algebroid $A$ is a Lie algebroid automorphism.\label{lemma:algebroidflow:morphism}
\end{lemma}
\begin{proof}
    Let $\alpha \in \Gamma(A)$ be a time-dependent vector field with adjoint flow $\Phi^t_\alpha: A \rightarrow A$. Let $\beta_1$ and  $\beta_2$ be two smooth sections of $A$.  In order to prove our result, we need to first show that the time-dependent section    \[F(t) := \Phi^t_\alpha[\beta_1, \beta_2]-[\Phi^t_\alpha\beta_1,\Phi^t\beta_2] \in \Gamma(A),\] is the zero section.

     Consider the differentiation of $F(t):$
\begin{align*}\frac{d}{dt}F(t) 
&= [\alpha,-\Phi^t_\alpha[\beta_1,\beta_2]]-[[\alpha,-\Phi^t_\alpha(\beta_1)], \Phi^t_\alpha(\beta_2)] -[\Phi^t_\alpha(\beta_1), [\alpha,-\Phi^t_\alpha(\beta_2)]]  \\
&= -[\alpha,\Phi^t_\alpha[\beta_1,\beta_2]]+[\alpha(t),[\Phi^t_\alpha\beta_1,\Phi^t_\alpha\beta_2]] \hspace{2cm} \text{(using Jacobi identity of algebroids)} \\
&= [\alpha, -\Phi^t_\alpha[\beta_1,\beta_2]+[\Phi^t_\alpha\beta_1, \Phi^t_\alpha\beta_1]]\\
&=-[ \alpha, F].
\end{align*}

Therefore \begin{equation*}
    \frac{d}{dt}F(t) = -[\alpha(t), F(t)] \qquad  \text{ and } \qquad F(0)= 0.\end{equation*}
From the uniqueness of solutions to adjoint flow equations, we conclude that $F(t) = 0$. 
 Hence, $\Phi^t_\alpha$ is a Lie algebra automorphism.

 Next, we want to show that $\Phi^t_\alpha$ is compatible with the anchor map, that is, $\Phi^t_{\rho{(\alpha)}}(\rho(\beta)) = \rho(\Phi^t_\alpha\beta)$, for any $\beta \in \Gamma(A)$. For this reason, we define a time-dependent smooth section of $A$,
 \[G(t):= \Phi^t_{\rho(\alpha)}(\rho(\beta))-\rho(\Phi^t_\alpha(\beta)).\]
    
 We differentiate it to get 
   \begin{align*}
\frac{d}{dt}G(t) 
&= -[\rho(\alpha(t)),\rho(\beta)]-\rho\Big(\frac{d}{dt}\Phi^t_\alpha\beta\Big)\\
&= -[\rho(\alpha(t)), \Phi^t_\alpha \rho(\beta)]-\rho([\alpha(t),\Phi^t_\alpha\beta]) \\
&= -[\alpha(t), \Phi^t_\alpha \rho(\beta)]- [\rho(\alpha(t)),\rho(\Phi^t_\alpha\beta)]\\
&= -[\rho(\alpha(t)), G(t)]
\end{align*}

Therefore, \[\frac{d}{dt}G(t) = -[\rho(\alpha(t)), G(t)], \hspace{1cm} G(0) = 0.\]

Using the uniqueness of the solution, we get $G(t) \equiv 0$.  Therefore, $\Phi^t_\alpha(\rho(\beta)) = \rho(\Phi^t_\alpha(\beta))$, for any $\beta \in \Gamma(A)$. Hence, $\Phi^t_\alpha$ is a Lie algebroid automorphism. 
\end{proof}
\section{\texorpdfstring{\( \Gamma(A) \)-path lemmas}{Gamma(A)-path lemmas}}

This section contains the technical lemmas used in section \ref{section:gamma(A)-homotopy} and section \ref{section:holonomy-of-groupoids}.They  concern complements of two-parameter families of sections, properties of slice algebroids, the flow product, and the adjoint flows. Together these results provide the foundational identities needed for the construction of the holonomy transformation groupoid and the algebroid holonomy groupoid.

\subsection{Lemmas about complements}

Throughout this subsection, if \[ \alpha(t,s) \in C^\infty([0,1]^2, \Gamma(A))\]
is a compactly supported two-parameter family of sections, we write \[ \beta(t,s)\]
for its complement and \[ \Phi({t,s}): \Gamma(A) \rightarrow \Gamma(A)\]
for its adjoint flow in the $t$-direction,  and  by \[\phi({t,s}):M \rightarrow M\]for the flow of $\rho(\alpha(t,s))$ in the $t$-direction.

The following proposition collects the basic properties of the complement construction used throughout Sections \ref{section:gamma(A)-homotopy} and \ref{section:holonomy-of-groupoids}. 
\begin{proposition}(Properties of complements) There exists a unique two-parameter family of sections $\beta(t,s) \in C^\infty([0,1]^2, \Gamma(A))$ satisfying \[ \frac{\partial\beta}{\partial t}- \frac{\partial \alpha}{\partial s} = [\beta,\alpha], \quad \quad \beta(0,s)  = 0.\] Moreover, 
\begin{enumerate}
    \item \label{lemma:complement.of.two.parameter.family} The complement is given by   \[ \beta(t,s) :=  \Phi({t,s})  \circ \int_0^t \Phi({u,s})^{-1} \frac{\partial \alpha}{\partial s} (u,s) du \in \Gamma(A).\]
    \item \label{lemma:flows.of.complements} If $\Psi_\beta^{(t,s)}$ denotes  the flow of $\beta(t,s)$ in the $s$-direction.Then 
    \begin{equation}
        \Psi{(t,s)}  \circ \Phi{(t,0)} = \Phi{(t,s)} \label{equation: psiphi equation}.
    \end{equation} 
    \item  \label{lemma:flows.of.complements.for.integral.curves} For all \( p_0 \in M \), let \( \gamma(t,s) \) denotes the integral curve of \( \rho(\alpha(t,s) \) with initial condition \( \gamma(0,s) = p_0 \). Then 
    \[ \frac{\partial \gamma}{\partial s} (t,s) = \rho(\beta(t,s))_{\gamma(t,s)}. \]
    \item \label{lemma:complement.of.flow.product.is.flow.product.of.complements}  If $ \alpha^\prime(t,s) $  is another two-parameter family of sections with complement $\beta^\prime$, then the complement of the flow product $\alpha \bullet \alpha^\prime$ is \[\beta(t,s) + \Phi^t_{\alpha} ( \beta'(t,s)).\] 
    
\end{enumerate}
\end{proposition}
\begin{proof}
    \begin{enumerate}
        \item We define the complement section of $\alpha$ as follows,
    \[ \beta(t,s) :=  \Phi({t,s})  \circ \int_0^t \Phi({u,s})^{-1} \frac{\partial \alpha}{\partial s} (u,s) du \in \Gamma(A).\]
    
A direct calculation shows that \( \beta(t,s) \) satisfies the desired initial value problem. Uniqueness follows from the uniqueness of the solution of the adjoint flow equation.

 \item    We denote the partial derivatives in $t$ and $s$-directions by  \( \partial_t \) and \( \partial_s\), respectively. The defining equations for each of \( \Phi \),   \( \alpha \),  and \( \beta \) are given as follows,
    \[ 
    \partial_t \Phi = [\Phi,\alpha ] \qquad \Phi(0,s) = \Id,
   \]
    and
    \[ \partial_t \beta -\partial_s \alpha = [\beta,\alpha]. \]
First we show that \( \partial_s \Phi = [\Phi,\beta]\). Consider the two-parameter family of endomorphisms of the Lie algebra
\[ Z(t,s) :=  [\Phi(t,s), \beta(t,s)] -  \partial_s \Phi (t,s). \]
We claim that \( Z(t,s) = 0 \). To see why, consider the time-derivative of \( Z \) while applying our defining equations, and the Jacobi identity:
\begin{align*}
    \partial_t Z &=  [\partial_t \Phi, \beta] + [\Phi, \partial_t \beta] -\partial_t \partial_s \Phi  \\
    &=   [[\Phi,\alpha],\beta] + [\Phi , \partial_t \beta] -\partial_s\partial_t \Phi\\
    &=  [[\Phi,\alpha],\beta]+[\Phi,\partial_t \beta] - \partial_s[\Phi,\alpha] \\
    &= [[\Phi,\alpha],\beta]+[\Phi,\partial_t \beta] - [\partial_s\Phi,\alpha] - [\Phi,\partial_s \alpha] \\
    &=[[\Phi,\alpha],\beta]-[\partial_s\Phi,\alpha] +[\Phi, \partial_t \beta - \partial_s \alpha] \\
    &=[[\Phi,\alpha],\beta] -[\partial_s\Phi , \alpha] +[\Phi,[\beta,\alpha]] \\
    &= [[\Phi,\beta],\alpha] - [\partial_s \Phi,\alpha] \\
    &=[[\Phi,\beta] - \partial_s \Phi,\alpha] \\
    &= [Z,\alpha]
\end{align*}
It is fairly straightforward to check that \( Z(0,s) = 0 \), and therefore \( Z \) satisfies the adjoint flow equation: 
\[ \partial_t Z = [Z,\alpha] \qquad Z(0,s) = 0   \]
By uniqueness of adjoint flows, we conclude that \( Z(t,s)  = 0 \) for all \( t\) and \(s \). Hence \( \partial_s \Phi = [\Phi,\beta]\).

To conclude the proof, we want to show   that the   smooth family of endomorphisms,
\[ Q(t,s) := \Psi(t,s) \Phi(t,0) - \Phi(t,s) \]
is $0$.  

Taking the derivative of $Q$ with respect to $s$, we obtain 
\[  \partial_s Q = [\Psi(t,s) \Phi(t,0) ,\beta] - [\Phi(t,s),\beta] = [Q,\beta].\]
Furthermore, observe that \( Q(t,0) = \Psi(t,0) \Phi(t,0) - \Phi(t,0) = 0\) since \( \Psi(t,0) = \Id \) so \( Q \) solves the adjoint flow equation:
\[  \partial_s Q = [Q,\beta] \qquad Q(t,0) = 0 \] 
Uniqueness of adjoint flows implies that \( Q = 0 \).

\item 
Let $\Psi(t,s)$ and $\psi(t,s)$ denote the adjoint flow of $\beta(t,s)$ and $\rho(\beta(t,s))$, respectively, in the $s$-direction.

    Applying the anchor map to the equation from  lemma \ref{lemma:flows.of.complements} gives
\[ \psi(t,s) \phi(t,0) = \phi(t,s). \]
Since \( \gamma(t,s) \) is the integral curve of \( X(t,s) \) with initial condition \( \gamma(0,s) = p_0 \), we have that \( \gamma(t,s) = \phi(t,s)(p_0) \). Therefore,
\[ \gamma(t,s) = \phi(t,s)(p_0) = \psi(t,s) \phi(t,0)(p_0) = \psi(t,s) (\gamma(t,0)) \]
Differentiating this equation with respect to \( s \) we obtain
\[ \frac{\partial \gamma}{\partial s} (t,s) = \rho(\beta(t,s))_{\psi(t,s)\left(\gamma(t,0)\right)} = \rho(\beta(t,s))_{\gamma(t,s)}.\]
This proves the Lemma.
\item 
For brevity, let us write \( F(t,s) \) to denote the pushforward along the adjoint flow of \( \alpha(t,s) \) in the \( t \)-direction. Hence,
\[ \alpha \bullet \alpha' = \alpha + F \alpha' \quad \text{ and } \quad Z = \beta + F \beta'. \]
From the definition of the complement, we must show that:
\[ \partial_t (\beta \bullet \beta') - \partial_s (\alpha \bullet \alpha') = [\beta \bullet \beta', \alpha \bullet \alpha'].\]
Expanding the left hand side (and suppressing the \( (t,s) \) dependence for readability) gives
\begin{align*}
    \partial_t Z - \partial_s (\alpha \bullet \alpha') 
    & = \partial_t (\beta + F \beta') - \partial_s (\alpha + F \alpha') \\ 
    & = \partial_t \beta + \partial_t (F \beta') - \partial_s \alpha - \partial_s (F \alpha') \\
    & = \partial_t \beta + (\partial_t F) \beta' + F (\partial_t \beta') - \partial_s \alpha - (\partial_s F) \alpha' - F (\partial_s \alpha').
\end{align*}
Recall that the equation \[ \partial_t F = [F,\alpha] \] is the  defining property of \( F \). Furthermore, it is straightforward to check that  Lemma~\ref{lemma:flows.of.complements} gives
\[ \partial_s F =  [F,\beta]. \]
Substituting the above two equations into the right-hand side of $\partial_tZ - \partial_s(\alpha\bullet \alpha^\prime)$  gives:
\begin{align*}
    & = \partial_t \beta + [F \beta' ,\alpha] + F (\partial_t \beta') - \partial_s \alpha - [F \alpha',\beta] - F (\partial_s \alpha') \\
    & = \partial_t \beta - \partial_s \alpha + F(\partial_t \beta' - \partial_s \alpha') + [F \beta',\alpha] - [F \alpha',\beta] \\
    & = [\beta,\alpha] + F [\beta',\alpha'] + [F \beta',\alpha] - [F \alpha',\beta] \\
    & = [\beta,\alpha] + [F\beta',F\alpha'] + [F \beta',\alpha] + [\beta, F \alpha'] \\
    & = [\beta + F \beta', \alpha + F \alpha'] \\
    & = [Z, \alpha \bullet \alpha']
\end{align*}

\end{enumerate}
\end{proof}

\subsection{Lemmas about slices}
\begin{lemma}\label{lemma:slice.algebroid.well.defined}
    Let \( S \subset M \) be a slice through \( x \in M \). Then \( A_S:=  \rho^{-1}(TS)\) is a vector bundle over \( S \) and inherits a Lie algebroid structure from \( A \).
\end{lemma}
\begin{proof}
    We will first show that \( A_S \) is a vector bundle over \( S \). To see this, observe that it can be thought of as a fiber product in the category of vector bundles:
   \[ \begin{tikzcd}
A_S \arrow[r] \arrow[d] & A \arrow[d, "\rho"] \\
TS \arrow[r, hook] & TM
    \end{tikzcd} \]
    The transversality condition says that the maps \( TS \to TM \) and \( \rho \colon A \to TM \) are transverse vector bundle maps, which is a well-known sufficient condition for the existence of such fiber products.

    To see why \( A_S \) inherits a Lie algebroid structure, we note that since \( S \) is an embedded submanifold, and section of \( A|_S \) can be extended to a section of \( A \) in a neighborhood of \( S \). Furthermore, if \( \alpha \) and \( \beta \) are sections of \( A \) with the property that \( \alpha|_S \in \Gamma(A_S) \) and \( \beta|_S \in \Gamma(A_S) \) then we have that \( \rho(\alpha) \) and \( \rho(\beta)\) are vector fields on \( M \) tangent to \( S \). Hence \( \rho([\alpha,\beta]) = [\rho(\alpha),\rho(\beta)] \) is tangent to \( S \), and therefore \( [\alpha,\beta]|_S \in \Gamma(A_S) \).
\end{proof}

\begin{lemma}\label{lemma:slice.algebroids.connected.by.trivial.holonomy}
    Let \( S \) and \( S' \) be two slices through the same point \( x\in M \) and write \( A_S \) and \( A_{S'} \) to denote the corresponding slice algebroids. There exists a trivial holonomy transformation \( \Theta \in \Alg_x (A_S, A_{S'}) \).
\end{lemma}

\begin{proof}
    From \cite{SingJoelAlf} we know that for a singular foliation \( \mathcal{F} \) on \( M \) and a pair of slices \( S \) and \( S' \) through the same point \( x \in M \), there exists an element \( X \in I_x \mathcal{F} \) such that the flow of \( X \) maps an open neighborhood of \( x \in S \) to an open neighborhood of \( x \in S' \).

    Letting \( \mathcal{F} = \rho(\Gamma(A)) \), we let \( \alpha \in I_x \Gamma(A) \) be a time-dependent section of \( A \) with the property that \( \rho(\alpha) = X \). The adjoint flow of a section is a Lie algebra isomorphism. Since it must preserve the anchor map, the adjoint flow of \( \alpha \) will have the property that it maps an open neighborhood of \( 0_x \in A_S \) to an open neighborhood of \( 0_x \in A_{S'} \).
\end{proof}

\subsection{Flow product lemmas}

\begin{lemma}\label{lemma:flow.product.compatible.with.flows}
    The flow of a flow-product is the flow-product of the flows, i.e.,
    \[ \Phi^t_{\alpha \bullet \beta} = \Phi^t_\alpha \Phi^t_\beta, \quad \text{ for all } t.\]
\end{lemma}
\begin{proof} Let $\eta = \alpha \bullet \beta.$
    Consider the time-dependent family of endomorphisms
    \[ Z(t) =  (\Phi^t_\eta)_* - (\Phi^t_\alpha)_* (\Phi^t_\beta)_*. \]
    Taking the derivative of \( Z \) with respect to \( t \) and apply the defining equations for the flows of \( \alpha \), \( \beta \), and \( \eta \), we get:
\begin{align*}
Z'(t) & = [\Phi^t_\eta, \eta] -  [\Phi^t_\alpha \Phi^t_\beta, \alpha] - \Phi^t_\alpha [\Phi^t_\beta, \beta]  \\
& = [\Phi^t_\eta,\eta] - [\Phi^t_\alpha \Phi^t_\beta, \alpha - \Phi^t_\alpha \beta]  \\
& = [\Phi^t_\eta, \eta] - [\Phi^t_\alpha \Phi^t_\beta, \alpha \bullet \beta ]  \\
& = [\Phi^t_\eta, \eta] - [\Phi^t_\alpha \Phi^t_\beta, \eta ]  \\
& = [Z(t), \eta].
\end{align*}
Since \( Z(0) = 0 \), it follows from uniqueness of solutions to the adjoint flow equation that \( Z(t) = 0 \), for all \( t \). Therefore, \( \Phi^t_\eta = \Phi^t_\alpha \Phi^t_\beta \), for all \( t \).
\end{proof}

\subsection{Lemmas about adjoint flows}
\begin{lemma}\label{lemma:flows.conjugate.by.automorphisms}
Let \( \alpha(t) \in \Gamma(A) \) be a time-dependent section of \( A \) and let \( F \colon A \to A \) be a Lie algebroid automorphism. Then
\[ F \circ \Phi^t_\alpha \circ F^{-1} = \Phi^t_{F_* \alpha}. \]
In other words, the flow of \( F_* \alpha \) is the conjugation of the flow of \( \alpha \) by \( F \).
\end{lemma}
\begin{proof}
    We will prove this by considering the equation at the level of pushforwards of sections. Computing the left hand side gives:
    \begin{align*}
\frac{d}{dt} F_* \circ (\Phi^t_\alpha)_* \circ (F^{-1})_* & = F_* \circ [(\Phi^t_\alpha)_*, \alpha] \circ (F^{-1})_*  \\
& = [F_* \circ (\Phi^t_\alpha)_* \circ (F^{-1})_*, F_* \alpha].
    \end{align*}
    This means that \( F_* \circ (\Phi^t_\alpha)_* \circ (F^{-1})_* \) satisfies the adjoint flow equation for \( F_* \alpha \). Since both of these flows are the identity at time \( t = 0 \), the lemma follows from the uniqueness of solutions to the adjoint flow equation.
\end{proof}

\begin{lemma}\label{lemma:holonomy.transformations.extend.to.automorphisms}
    Suppose \( F \in \Alg_x(A_S, A_{S'}) \) is a holonomy transformation. Then there exists a (locally defined) Lie algebroid automorphism \( \tilde F \colon A \to A \) with the property that \( \tilde F|_{A_S} = F \).
\end{lemma}
\begin{proof}
    We utilize a local splitting theorem around slices. From \cite{Rui+holonomy+Charactclasses+2005},  we know that for any slice \( S \) through a point \( x \in M \),  there exists an open neighborhood \( U \subset M \) of \( x \) and a Lie algebroid isomorphism
\[  A|_U \to A_S \times T\R^k, \]
where \( k  \) is the dimension of the leaf through \( x \) of the characteristic foliation of \( A \).

Therefore, in an open neighborhood of \( x \), given a holonomy transformation \( F \in \Alg_x(A_S, A_{S'}) \) we can define a Lie algebroid automorphism \( \tilde F \colon A \to A \) by taking:
\[ \tilde F(a,v) = (F(a), v). \]
\end{proof}

\begin{lemma}\label{lemma:trivial.holonomy.transformations.are.normal.subgroupoid}
The subgroupoid \( \THT(A) \subset  \FHT(A) \) of trivial holonomy transformations is a normal subgroupoid. The orbits of \( \THT(A) \) are precisely the sets of all slice algebroids through a given point.
\end{lemma}
\begin{proof}
The latter claim follows immediately from Lemma~\ref{lemma:slice.algebroids.connected.by.trivial.holonomy}. 

To prove  \( \THT(A) \) is a normal subgroupoid, we need to show that  \[ F^{-1} \circ \Theta  \circ F \in \THT(A),   \quad  \text{ for all } \Theta \in \THT(A)  \text{ and }  F \in \FHT(A)  .\] 

Fix a $ \Theta \in \THT(A)$ and $ F \in \FHT(A)$. Since $ \Theta \in \THT(A)$, there exists a time-dependent section \( \alpha(t) \in I_x \Gamma(A) \) with the property that \( (\Phi^1_\alpha)|_{A_S} = \Theta\). Moreover, Lemma \ref{lemma:holonomy.transformations.extend.to.automorphisms}  gives us a Lie algebroid automorphism \( \tilde F \colon A \to A \) with the property that \( \tilde F|_{A_S} = F \). Hence
\[ F \circ \Theta|_{A_S} \circ F^{-1} = \til F \circ \Theta \circ \til F^{-1}|_{A_{S'}} \]

From Lemma~\ref{lemma:flows.conjugate.by.automorphisms}, we obtain
\[ \til{F} \circ \Theta|_{A_S} \circ \til{F}^{-1} = \til{F} \circ (\Phi^1_\alpha)|_{A_S} \circ \til{F}^{-1} = (\Phi^1_{\til{F}_* \alpha}).\]
This establishes that \(F  \circ \Theta|_{A_S} \circ F^{-1} \) is a trivial holonomy transformation.

\end{proof}

\section{Homotopy Lemmas}
The purpose of this section is to compare the notion of $\Gamma(A)$-homotopy  introduced in section \ref{section:gamma(A)-homotopy} with the classical notion of $A$-homotopy of Crainic-Fernandes. We briefly recall the latter and then show that the two notions coincide.

A variation of $A$-paths is a map \[a_s(t)=a(t,s): I \times I \rightarrow A\] such that $a_s$ is a family of $A$-paths of class $C^2$ on $s$, with the property that the base paths \[\gamma_s(t) = \gamma(t,s): I \times I \rightarrow M\] have fixed end points.  Given a connection   $\nabla$  on $A$,  and a variation of $A$-paths $a(t,s)$ with base path $\gamma(t,s)$, define  \[ \nabla_t(a);= \left(\nabla_{X} \alpha + \frac{\partial\alpha}{ \partial t}\right)_{\gamma(t,s)},  \]
where $\alpha(t,s)$ is a two-parameter family of sections extending $a(t,s)$   and $X(t,s)$ is a time-dependent vector field extending $\partial_t\gamma$. Similarly one defines $\nabla_s(a)$.

The  $A$-torsion of $\nabla$ is \[ T_\nabla(\alpha, \beta) = \nabla_{\rho(\alpha)} \beta - \nabla_{\rho(\beta)} \alpha + [\beta,\alpha].\]

The complement of $a(t,s)$ is the unique $A$-valued function $b(t,s)$ along $\gamma(t,s)$ satisfying \[ \nabla_t b -\nabla_s a = T_\nabla(a,b), \quad \quad b(0,s) = 0.\]

The variation $a(t,s)$ is called an $A$-homotopy if \[ b(1,s) = 0,  \quad \forall  s \in [0,1].\]

Our next lemma will points out that, by working at the level of sections, one can immediately see that complement of an $A$-path always exists and does not depend on the choice of connection.
\begin{lemma}\label{lemma:complement-of-extension}
    Suppose $a(t,s)$ is a variation of $A$-paths and $\alpha(t,s) \in \Gamma(A)$ is two-parameter section extending $a(t,s)$. Suppose that $\beta(t,s) \in \Gamma(A)$ is the complement of $\alpha(t,s)$. Then:
    \[ b(t,s) := \beta(t,s)|_{\gamma(t,s)}\]
    is the complement of $a(t,s)$.
\end{lemma}
\begin{proof}
    By definition, since $\beta$ is the complement of $\alpha$ we have that:
    \[ \frac{\partial \beta}{\partial t} - \frac{\partial \alpha}{\partial s} = [\beta,\alpha].\]
    Adding $\nabla_{\rho(\alpha)} \beta - \nabla_{\rho(\beta)} \alpha$ to both sides yields:
    \[ \nabla_{\rho(\alpha)} \beta + \partial_t \beta - \nabla_{\rho(\beta)}  \alpha -\partial_s \alpha = \nabla_{\rho(\alpha)} \beta - \nabla_{\rho(\beta)} \alpha +  [\beta,\alpha].\]
    Now, since $a(t,s)$ is a variation,  it follows that $\rho(\alpha)$ extends $\partial\gamma/\partial t$. By lemma \ref{lemma:flows.of.complements.for.integral.curves}, we also know that $\rho(\beta) = \frac{\partial \gamma}{\partial s}$. 
    Therefore, it follows that:
    \[ \nabla_t b - \nabla_s a = T_\nabla(a,b).\]
\end{proof}
An important consequence of this fact is an equivalence between the notion of $A$-homotopy and $\Gamma(A)$-homotopy.
\begin{lemma} \label{lemma:A:homotopypaths:Gamma(A)paths}
    Suppose \( a_0(t) \) and \( a_1(t) \) be two  \( A \)-paths with the same initial point $x_0$.  Let $\alpha_0(t)$ and $\alpha_1(t)$ be arbitrary extensions of extensions of $a_0$ and $a_1$, respectively. Then \[(\alpha_0,x_0) \text{ and } (\alpha_1,x_0) \text{ are } \Gamma(A) \text{ - homotopic }\] if and only if \[(a_0,x_0) \text{ and } (a_1,x_0) \text{ are } A \text{ - homotopic. }\]
    Consequently, the $\Gamma(A)$-homotopy relation coincides with the classical Crainic-Fernandes $A$-homotopy relation.
\end{lemma}
\begin{proof}

Assume first that 
 $a_0(t)$ and $a_1(t)$ are two  $A$-paths with $\Gamma(A)$-path extensions $\alpha_0(t)$ and $\alpha_1(t)$, respectively. Let $a(t,s)$ be a homotopy between them with underlying family of paths $\gamma(t,s)$.

Choose a compactly supported, two-parameter family of sections $\alpha(t,s)$ extending $a(t,s)$  satisfying 

\[ \alpha(t,0) = \alpha_0(t) \qquad \alpha(t,1) = \alpha_1(t). \]
This family defines a path in the space of $\Gamma(A)$-paths:
\[ s \mapsto (\alpha(\bullet,s),\gamma(0,s) ).\]
Furthermore, $\alpha(t,s)$ has a complement $\beta(t,s)$,  and by Lemma~\ref{lemma:complement-of-extension}, we know that the $A$-path complement of $a(t,s)$ is given by:
\[ b(t,s) = \beta(t,s)|_{\gamma(t,s)}. \]
Since $a(t,s)$ is a homotopy we know $b(1,s) = 0 $ and hence:
\[ \beta(1,s)|_{\gamma(1,s)} = 0. \]
But this is precisely the condition that $(\alpha(t,s),\gamma(0,s))$ is a homotopy.

Conversely, suppose $(\alpha_0,x_0)$ and $(\alpha_1, x_0)$ are $\Gamma(A)$-homotopic. Let $\alpha(t,s)$ be a homotopy between them and let $\beta(t,s)$ be its complement.  Define:
\[ \gamma(t,s) := \phi_{\rho(\alpha(-,s))}^t(x_0) \]
and  \[a(t,s) := \alpha(t,s)_{\gamma(t,s)}.\] Then $a(t,s)$ is a variation of $A$-paths whose boundary paths are $a(t,0) = a_0(t)$ and  $a(t,1) = a_1(t)$.  

By lemma~\ref{lemma:complement-of-extension},
$$ b(t,s) := \beta(t,s)|_{\gamma(t,s)} $$
is the complement of $a(t,s)$. Furthermore, since $\alpha(t,s)$ is assumed to be a homotopy, it follows that
\[ b(1,s) = \beta(1,s)|_{\gamma(1,s)} = 0. \]
Hence, $a(t,s)$ is an $A$-homotopy. Therefore, the two notions of homotopy coincide.
\end{proof}
There is another way of defining ``multiplication'' for $\Gamma(A)$-paths which is closer to the classical one used in the construction of the Weinstein groupoid. 
\begin{definition}\label{definition:concatenation}
Suppose $\alpha$ and $\beta$ are time-dependent elements of $\Gamma(A)$. Let $\rho \colon [0,1] \to [0,1]$ be a smooth function with $\rho(0) = 0$, $\rho(1) = 1$ and such that all derivatives vanish near the endpoints. 
The \emph{$\rho$-concatenation} of $(\alpha, p)$ and $(\beta, q)$ is the $\Gamma(A)$-path defined by
\footnote{The idea behind this definition is that we first reparameterize the paths to fit into each half of the interval and then concatenate. The constant multiples appear due to the chain rule and the fact that our paths are already in the space of ``derivatives.''.}:
\[
\alpha \ast \beta (t) = \begin{cases}
   2 \rho'(2t) \beta(\rho(2t))  & \text{if } t \in [0,1/2],\\
   2 \rho'(2t-1) \alpha(\rho(2t-1)) & \text{if } t \in [1/2,1].
\end{cases}
\]
The $\rho$-concatenation of two composable $\Gamma(A)$-paths $(\alpha, p)$ and $(\beta, q)$ is then defined as the $\Gamma(A)$-path $(\alpha \ast \beta, p)$.
\end{definition}
A standard argument shows that the $\rho$-concatenation of $\Gamma(A)$-paths is independent of the choice of $\rho$ up to $\Gamma(A)$-homotopy.
Furthermore, the groupoid axioms do not hold strictly for the $\rho$-concatenation but do hold up to ``natural'' homotopies. 
However, one advantage of this definition is that it can be directly applied to $A$-paths by applying the definition pointwise. 
The $A$-paths version of this definition is precisely the one that is traditionally used to define the Weinstein groupoid.
\begin{lemma}\label{lemma:Gamma(A)homotopy-Adjoint-homotopy}

Let $A$ be a Lie algebroid. Then, the $\rho$-concatenation of $\Gamma(A)$-paths is homotopic to the flow product of $\Gamma(A)$-paths for any choice of reparameterization $\rho$. 
    
\end{lemma}

\begin{proof}
Suppose $(\alpha, p)$ and $(\beta, q)$ are composable $\Gamma(A)$-paths. Write $(\alpha, p) \ast (\beta,q)$ to denote the $\rho$-concatenation of these $\Gamma(A)$-paths relative to some suitable reparameterization $\rho$.

We need to show that $(\alpha, p) \ast (\beta,q)$ is $\Gamma(A)$-homotopic to $(\alpha, p) \cdot (\beta,q)$, the flow product of these $\Gamma(A)$-paths.
For both the $\rho$-concatenation and the flow product, performing a homotopy of one of the paths results in a homotopy of the resulting path (see Proposition~\ref{proposition:flow.product.compatible.with.homotopy} for the flow product version).
This means, without loss of generality, we can replace each of $(\alpha,p)$ and $(\beta,q)$ with $\Gamma(A)$-homotopic paths. 
Since different choices of $\rho$ result in $\Gamma(A)$-homotopic paths we can also choose $\rho$ however we like for the purposes of the argument.

Select $\rho(t)$ to be one-to-one and such that there exists an open subinterval $U \subset [0,1]$ on which $\rho|_{U} = \Id$.
Choose $\alpha$ and $\beta$ to be such that $\alpha(t) = 0$ and $\beta(t)=0$ for $t$ outside $U$.

Define:
\[   \overline{\alpha}(t) := \begin{cases}
    0 & \text{if } t \in [0,1/2], \\
    2 \alpha(2t-1) & \text{if } t \in [1/2,1],
\end{cases} 
\qquad 
\overline{\beta}(t) := \begin{cases}
    2 \beta(2t) & \text{if } t \in [0,1/2],\\
    0 & \text{if } t \in [1/2,1].
\end{cases} \]
Note that each of $\overline{\alpha}$ and $\overline{\beta}$ can be obtained from $\alpha$ and $\beta$ by reparameterization of the interval $[0,1]$ and hence are $\Gamma(A)$-homotopic to $\alpha$ and $\beta$, respectively 
(see Example~\ref{example:reparameterization.as.homotopy} for why reparameterizations are homotopies).

Since the flow product is compatible with homotopy (Proposition~\ref{proposition:flow.product.compatible.with.homotopy}) we conclude that $\alpha \bullet \beta$ is $\Gamma(A)$ homotopic to the flow product of $\overline{\alpha} \bullet \overline{\beta}$. 
However, from the definition of the flow product and the fact that $\overline{\alpha}(t) = 0 $ for $t \in [0,1]$ and $\overline{\beta}(t) = 0 $ for $t \in [1/2,1]$ we conclude:
$$ \overline{\alpha} \bullet \overline{\beta}(t) = \overline{\alpha}(t) + \overline{\beta}(t).  $$

On the other hand, the $\rho$-concatenation of $\alpha$ with $\beta$ is given by:
\[ \alpha \ast \beta (t) = \begin{cases}
   2 \rho'(2t) \beta(\rho(2t))  & \text{if } t \in [0,1/2],\\
   2 \rho'(2t-1) \alpha(\rho(2t-1)) & \text{if } t \in [1/2,1].
\end{cases}\]
Since $\rho|_U = \Id_U$ and $\alpha$ and $\beta$ are supported in $U$, this expression simplifies to:
\[ \alpha \ast \beta (t) = \begin{cases}
   2\beta(2t)  & \text{if } t \in [0,1/2],\\
   2\alpha(2t-1) & \text{if } t \in [1/2,1].
\end{cases}
\]
Therefore,
\[ \alpha \ast \beta (t) = \overline\alpha(t) + \overline{\beta}(t) =\overline{\alpha} \bullet \overline{\beta}(t). \]
But the right hand side is $\Gamma(A)$-homotopic to the flow product $\alpha \bullet \beta$ as we have shown above. Therefore, we conclude that the $\rho$-concatenation of $\Gamma(A)$-paths is homotopic to the flow product of $\Gamma(A)$-paths for any choice of reparameterization $\rho$.
\end{proof}

\section{Smoothness Lemmas}

\begin{lemma}\label{Lemma:adjoint-flows-are-equal}
    Suppose \( \alpha(t,s) \in \Gamma(A) \) is a two-parameter vector family of sections and \( \beta(t,s) \) is the complement. If \( \beta(1,s) = 0 \) for all \( s \), then the time-1 adjoint flows of \( \alpha(t,0) \) and \( \alpha(t,1) \) are equal.
\end{lemma}
\begin{proof}

 Let $\Psi^{(t,s)}$ be the flow of $\beta(t,s)$ in the $s$-direction and let $\Phi^{(t,s)}$ be the flow of $\alpha(t,s)$ in the $t$-direction. We already know from lemma \ref{lemma:flows.of.complements} that 
    \begin{equation}\label{aboveequation}
        \Psi(t,s) \Phi(t,0) = \Phi(t,s). 
    \end{equation} 
    Setting $t = 1,$ we get \[    \Psi(1,s) \Phi(1,0) = \Phi(1,s).  \]

    Since $\Psi$ is the $s$-flow of $\beta$, it satisfies \[ \frac{\partial}{\partial s} \Psi(1,s) = [\Psi(1,s), \beta(1,s)].\] By hypothesis, $\beta(1,s) = 0$ for all $s$, hence \[ \frac{\partial}{\partial s} \Psi(1,s) = 0.\]

    Together with the initial condition $\Psi(1,0) = \mathrm{Id}, $ this implies \[ \Psi(1,s) = \mathrm{Id}, \quad \quad \forall s \in [0,1],\] and in particular
 \[\Phi(1,0) = \Phi(1,1).\]

\end{proof}

\begin{lemma}\label{conjugacy-smooth-local-section-family}
    Let $\G$ be a sufficiently small local integration satisfying Lemma \ref{lemma:conjugacy-class-embedded}. Let $S_x$ be a slice through $x$, and let \[ \sigma: S_x \rightarrow Z(\G)\]
be a smooth local section such that $\sigma(y) \in \G_y$ for all $y \in S_x.$ Then, after shrinking around $x$ if necessary, there exists a smooth local section \[ \tilde{\sigma}: U \rightarrow Z(\G)\]
defined on an open neighborhood $U \subset M$ of $x,$ such that \[ \tilde{\sigma}|_{U \cap S_x} = \sigma|_{U \cap S_x}\]
     and \[ im(\tilde{\sigma}) \subset \bigcup_{y \in S_x} C(\sigma(y)).\]
     Moreover, the germ of the union \[ \bigcup_{y \in S_x} C(\sigma(y))\]

near $\sigma(S_x)$ is the image of this section.

\end{lemma}

\begin{proof}
    Shrink $S_x$ if necessary to choose local sections \[ a_1,...,a_k \in \Gamma(A)\]
such that the vector fields \[ X_i:= \rho(a_i)\]
span a complement to $TS_x$ along $S_x.$ By the inverse function theorem, after shrinking $S_x$ and the domains of the flows, the map 
\[\Psi:S_x \times \mathbb{R}^k \rightarrow M\]
    defined by \[ \Psi(y, t_1,...,t_k) = \phi^{t_k}_{X_k} \circ \cdot \cdot \cdot \circ \phi^{t_1}_{X_1}(y)\]

is a diffeomorphism from a neighborhood of $(x,0)$ onto an open neighborhood \[U \subset M\]

of $x.$

    For each $i$, let $\tau_i^{t_i}$ denote the local bisection of $\G$ integrating the section $t_ia_i$, in the sense of definition \ref{definition:G-integrable-section}. By shrinking the local integration if necessary, and using Lemma \ref{lemma:local-algebra-principal}, we may assume that all products and inverses below are defined. Set \[ \tau^t:= \tau_k^{t_k} \cdot \cdot \cdot \tau_1^{t_1},\]
with the product ordered so that \[s(\tau^t(y))=y, \quad \quad t(\tau^t(y))=\Psi(y,t).\]
Define \[ \tilde{\sigma}:U \rightarrow \G\]
as follows. Given $p \in U,$ write uniquely \[ p = \Psi(y,t)\]
with $y \in S_x$ and $t = (t_1,...,t_k)$ small. Then set \[ \tilde{\sigma}(p):= \tau^t(y) \sigma(y) \tau^t(y)^{-1}.\]
This is well-defined by uniqueness of the coordinates $(y,t),$ and it is smooth because $\Psi,$ the bisections $\tau^t$, the section $\sigma,$ and the local groupoid operations are smooth. 
Since $\sigma(y) \in \G_y$ and $\tau^t(y) : y \rightarrow p$, we have \[ s(\tilde{\sigma}(p)) = p = t(\tilde{\sigma}(p)).\]
    Thus \[ \tilde{\sigma}(p) \in \G_p.\]

We next show that $\tilde{\sigma}(p) \in Z(\G_p).$ Let $h \in \G_p$ be an isotropy element for which the relevant products are defined. Since $\tau^t(y): y \rightarrow p, $ the element \[ (\tau^t(y))^{-1} h \tau^t(y)\]

    lies in $\G_y$. Since $\sigma(y) \in Z(\G_y), $ we have 

\[\sigma(y) (\tau^t(y))^{-1} h \tau^t(y) = (\tau^t(y))^{-1} h \tau^t(y) \sigma(y). \]

Conjugating by $\tau^t(y),$ we obtain \[ \tilde{ \sigma}(p)h = h \tilde{ \sigma}(p).\]
 Therefore, $\tilde{\sigma}(p) \in Z(\G_p).$
     Hence \[ \tilde{\sigma}:U \rightarrow Z(\G)\]

is a smooth local section.

     If $p \in U \cap S_x,$ then $p = \Psi(p, 0)$ and $\tau^0(p) = 1_p.$ Hence \[ \tilde{\sigma}(p) = 1_p \sigma(p) 1_p^{-1} = \sigma(p),\]
so \[ \tilde{\sigma}|_{U \cap S_x} = \sigma|_{U \cap S_x}.\]
By construction, every value of $\tilde{\sigma}$ is obtained by conjugating some $\sigma(y).$ Therefore, \[ im(\tilde{\sigma}) \subset \bigcup_{y \in S_x} C(\sigma(y)).\]
It remains to identify the germ of this union near $\sigma(S_x).$ Let 
 \[ q \in \bigcup_{y \in S_x} C(\sigma(y))\]
be sufficiently close to $\sigma(S_x).$ Then, after shrinking if necessary, we may write \[ q = h \sigma(y) h^{-1}\]
for some $y \in S_x$ and some arrow $h: y \rightarrow p,$ with $p \in U.$ Since $p \in U$, write
\[p  = \Psi(y^\prime,t)\]
using the local product coordinates. Shrinking $U$ once more if necessary, the local product chart implies that $h: y \rightarrow p$ is sufficiently close to the unit section and $y \in S_x,$ then $y = y^\prime.$ Thus \[ p = \Psi(y,t).\]
     Now, both \[ h: y \rightarrow p \quad \text{ and } \tau^t(y): y \rightarrow p\]

are arrows with the same source and target. Hence \[ (\tau^t(y))^{-1} h \in \G_y.\]
Since  $\sigma(y) \in Z(\G_y)$, conjugation by $(\tau^t(y))^{-1}h$ fixes $\sigma(y).$ Therefore \[ h \sigma(y) h^{-1} = \tau^t(y) \sigma(y) \tau^t(y)^{-1} = \tilde{\sigma}(p).\]
     Hence the local germ of \[ \bigcup_{y \in S_x} C(\sigma(y))\]

near $\sigma(S_x)$ is contained in $im(\tilde{\sigma})$. Together with the previous inclusion, this proves that the germ of the union is precisely the image of $\tilde{\sigma.}$

Finally, Lemma \ref{lemma:conjugacy-class-embedded} ensures that, after the above shrinking, the relevant local conjugacy classes are embedded and meet each nearby isotropy fiber in at most one point. Thus the germ of the union is the graph of the smooth local section \[\tilde{\sigma}: U \rightarrow Z(\G).\]

\end{proof}

\begin{proposition}\label{lemma:Lie-grpd-has-holonomy}
     Every classical Lie groupoid has holonomy. 
\end{proposition}

\begin{proof}

Let $g \in \calG$ be arbitrary.
We must show that given two bisections $\sigma_1$ and $\sigma_2$ are two local bisections  such that  \[ \sigma_1(x) = \sigma_2(x) = g\]
then we have that the holonomy transformation induced by $\Ad_{\sigma_1}$ is equal to the holonomy transformation induced by $\Ad_{\sigma_2}$. Let $\eta = \sigma_1^{-1} \sigma_2$. Then $\eta$ is a bisection with $\eta(x) = 1_x$. 
It suffices to show that $Ad_{\eta}$ induces a trivial holonomy transformation.

By Lemma~\ref{Lemma:open-restriction-has-holonomy} we know that there must exist an open restriction $\calG^\circ$ of $\calG$ where the holonomy transformation associated to any bisection through an identity element must be trivial. 
In particular, there is a neighborhood $U$ of $x$ where $\eta|_{U}$ is a local bisection contained in $\calG^\circ$ and so $\Ad_{\eta|_{U}}$ must induce a trivial holonomy transformation.

\end{proof}
\printbibliography

@article {Androulidakis+Zambon,
    AUTHOR = {Androulidakis, Iakovos and Zambon, Marco},
     TITLE = {Holonomy transformations for singular foliations},
   JOURNAL = {Adv. Math.},
  FJOURNAL = {Advances in Mathematics},
    VOLUME = {256},
      YEAR = {2014},
     PAGES = {348--397},
      ISSN = {0001-8708,1090-2082},
   MRCLASS = {53C12 (22A25 53C29 58H05)},
  MRNUMBER = {3177296},
MRREVIEWER = {Liviu\ Octavian\ Popescu},
       DOI = {10.1016/j.aim.2014.02.003},
       URL = {https://doi-org.proxyiub.uits.iu.edu/10.1016/j.aim.2014.02.003},
}

@book {Lie-groups-Duistermaat,
    AUTHOR = {Duistermaat, J. J. and Kolk, J. A. C.},
     TITLE = {Lie groups},
    SERIES = {Universitext},
 PUBLISHER = {Springer-Verlag, Berlin},
      YEAR = {2000},
     PAGES = {viii+344},
      ISBN = {3-540-15293-8},
   MRCLASS = {22Exx (22-01 22C05 43-01)},
  MRNUMBER = {1738431},
MRREVIEWER = {Aloysius\ Helminck},
       DOI = {10.1007/978-3-642-56936-4},
       URL = {https://doi.org/10.1007/978-3-642-56936-4},
}

@book {ReebStability,
    AUTHOR = {Reeb, Georges},
     TITLE = {Sur certaines propri\'et\'es topologiques des vari\'et\'es
              feuillet\'ees},
    SERIES = {Publications de l'Institut de Math\'ematiques de
              l'Universit\'e{} de Strasbourg [Publications of the
              Mathematical Institute of the University of Strasbourg]},
    VOLUME = {11},
      NOTE = {Actualit\'es Scientifiques et Industrielles, No. 1183.
              [Current Scientific and Industrial Topics]},
 PUBLISHER = {Hermann \& Cie, Paris},
      YEAR = {1952},
     PAGES = {5--89, 155--156},
   MRCLASS = {56.0X},
  MRNUMBER = {55692},
MRREVIEWER = {H.\ Samelson},
}

@article {ThurstonStability,
    AUTHOR = {Thurston, William P.},
     TITLE = {A generalization of the {R}eeb stability theorem},
   JOURNAL = {Topology},
  FJOURNAL = {Topology. An International Journal of Mathematics},
    VOLUME = {13},
      YEAR = {1974},
     PAGES = {347--352},
      ISSN = {0040-9383},
   MRCLASS = {57D30 (57A35)},
  MRNUMBER = {356087},
MRREVIEWER = {Robert\ Roussarie},
       DOI = {10.1016/0040-9383(74)90025-1},
       URL = {https://doi-org.proxyiub.uits.iu.edu/10.1016/0040-9383(74)90025-1},
}

@article {Fernandes-Daan-Associativity-Integrability,
    AUTHOR = {Fernandes, Rui Loja and Michiels, Daan},
     TITLE = {Associativity and integrability},
   JOURNAL = {Trans. Amer. Math. Soc.},
  FJOURNAL = {Transactions of the American Mathematical Society},
    VOLUME = {373},
      YEAR = {2020},
    NUMBER = {7},
     PAGES = {5057--5110},
      ISSN = {0002-9947,1088-6850},
   MRCLASS = {58H05 (22A22 22E05 53D17)},
  MRNUMBER = {4127871},
MRREVIEWER = {Mar\'ia\ Amelia\ Salazar},
       DOI = {10.1090/tran/8073},
       URL = {https://doi.org/10.1090/tran/8073},
}

@article {Cabrera-Ioan-Maria-local-integration-Lie-brackets,
    AUTHOR = {Cabrera, Alejandro and M\u arcu\c t, Ioan and Salazar, Mar\'ia
              Amelia},
     TITLE = {On local integration of {L}ie brackets},
   JOURNAL = {J. Reine Angew. Math.},
  FJOURNAL = {Journal f\"ur die Reine und Angewandte Mathematik. [Crelle's
              Journal]},
    VOLUME = {760},
      YEAR = {2020},
     PAGES = {267--293},
      ISSN = {0075-4102,1435-5345},
   MRCLASS = {58H05 (17B99 22A22)},
  MRNUMBER = {4069892},
MRREVIEWER = {Matias\ L.\ del Hoyo},
       DOI = {10.1515/crelle-2018-0011},
       URL = {https://doi.org/10.1515/crelle-2018-0011},
}

@article{Rui+holonomy+Charactclasses+2005,
    AUTHOR = {Fernandes, Rui Loja},
     TITLE = {Lie algebroids, holonomy and characteristic classes},
   JOURNAL = {Adv. Math.},
  FJOURNAL = {Advances in Mathematics},
    VOLUME = {170},
      YEAR = {2002},
    NUMBER = {1},
     PAGES = {119--179},
      ISSN = {0001-8708,1090-2082},
   MRCLASS = {58H05 (53C05 53C12 53C29 53D17 57R20 57R30)},
  MRNUMBER = {1929305},
MRREVIEWER = {Jan\ Kubarski},
       DOI = {10.1006/aima.2001.2070},
       URL = {https://doi.org/10.1006/aima.2001.2070},
}

@incollection {Rui+Marius,
    AUTHOR = {Crainic, Marius and Fernandes, Rui Loja},
     TITLE = {Lectures on integrability of {L}ie brackets},
 BOOKTITLE = {Lectures on {P}oisson geometry},
    SERIES = {Geom. Topol. Monogr.},
    VOLUME = {17},
     PAGES = {1--107},
 PUBLISHER = {Geom. Topol. Publ., Coventry},
      YEAR = {2011},
   MRCLASS = {53D17 (22A22 58H05)},
  MRNUMBER = {2795150},
MRREVIEWER = {Chenchang\ Zhu},
       DOI = {10.2140/gt},
       URL = {https://doi.org/10.2140/gt},
}

@article {SingJoelAlf,
    AUTHOR = {Garmendia, Alfonso and Villatoro, Joel},
     TITLE = {Integration of singular foliations via paths},
   JOURNAL = {Int. Math. Res. Not. IMRN},
  FJOURNAL = {International Mathematics Research Notices. IMRN},
      YEAR = {2022},
    NUMBER = {23},
     PAGES = {18401--18445},
      ISSN = {1073-7928,1687-0247},
   MRCLASS = {53C12},
  MRNUMBER = {4519148},
MRREVIEWER = {Sylvain\ Lavau},
       DOI = {10.1093/imrn/rnab177},
       URL = {https://doi.org/10.1093/imrn/rnab177},
}

@article {Androulidakis+Skandalis,
    AUTHOR = {Androulidakis, Iakovos and Skandalis, Georges},
     TITLE = {The holonomy groupoid of a singular foliation},
   JOURNAL = {J. Reine Angew. Math.},
  FJOURNAL = {Journal f\"ur die Reine und Angewandte Mathematik. [Crelle's
              Journal]},
    VOLUME = {626},
      YEAR = {2009},
     PAGES = {1--37},
      ISSN = {0075-4102,1435-5345},
   MRCLASS = {58H05 (46L80 57R30 58J42)},
  MRNUMBER = {2492988},
MRREVIEWER = {Yuri\ A.\ Kordyukov},
       DOI = {10.1515/CRELLE.2009.001},
       URL = {https://doi.org/10.1515/CRELLE.2009.001},
}

@article {LongtitudinalSmoothnesDebord,
    AUTHOR = {Debord, Claire},
     TITLE = {Longitudinal smoothness of the holonomy groupoid},
   JOURNAL = {C. R. Math. Acad. Sci. Paris},
  FJOURNAL = {Comptes Rendus Math\'ematique. Acad\'emie des Sciences. Paris},
    VOLUME = {351},
      YEAR = {2013},
    NUMBER = {15-16},
     PAGES = {613--616},
      ISSN = {1631-073X,1778-3569},
   MRCLASS = {53C12 (53C29)},
  MRNUMBER = {3119886},
MRREVIEWER = {Ryad\ Ghanam},
       DOI = {10.1016/j.crma.2013.07.025},
       URL = {https://doi.org/10.1016/j.crma.2013.07.025},
}

@article {IntegrabilityofLiebrackets,
    AUTHOR = {Crainic, Marius and Fernandes, Rui Loja},
     TITLE = {Integrability of {L}ie brackets},
   JOURNAL = {Ann. of Math. (2)},
  FJOURNAL = {Annals of Mathematics. Second Series},
    VOLUME = {157},
      YEAR = {2003},
    NUMBER = {2},
     PAGES = {575--620},
      ISSN = {0003-486X,1939-8980},
   MRCLASS = {58H05 (17B37 17B99 22A22 53D17)},
  MRNUMBER = {1973056},
MRREVIEWER = {Johannes\ Huebschmann},
       DOI = {10.4007/annals.2003.157.575},
       URL = {https://doi.org/10.4007/annals.2003.157.575},
}
\end{document}